\documentclass[leqno]{article}

\usepackage[frenchb,english]{babel}
\usepackage[T1]{fontenc}
\usepackage{amsmath}
\usepackage{amssymb}
\usepackage{amsfonts}
\usepackage{enumerate}
\usepackage{vmargin}
\usepackage[all]{xy}
\usepackage{mathrsfs}
\usepackage{mathtools}
\usepackage{lmodern}
\usepackage[colorlinks=true,pagebackref=true]{hyperref}%colorlinks=true,urlcolor=blue
\usepackage{comment}
\setmarginsrb{3cm}{3cm}{3.5cm}{3cm}{0cm}{0cm}{1.5cm}{3cm}
\mathchardef\sdash="2D

\newcommand{\E}{\ensuremath{\mathbb{E}}}
\newcommand{\N}{\ensuremath{\mathcal{N}}}
\newcommand{\nn}{\ensuremath{\mathbb{N}}}
\newcommand{\B}{\mathrm{B}} %Bounded operators
\newcommand{\I}{\mathrm{I}} %matrice identité

\let\L\relax %\L ne fait plus rien !
\newcommand{\L}{\mathrm{L}}

\newcommand{\M}{\mathrm{M}}

\newcommand{\R}{\ensuremath{\mathbb{R}}}
\newcommand{\C}{\ensuremath{\mathbb{C}}}

\newcommand{\Id}{{\rm Id}}

\renewcommand{\leq}{\ensuremath{\leqslant}}
\renewcommand{\geq}{\ensuremath{\geqslant}}
\newcommand{\n}{\noindent}
\newcommand{\qed}{\hfill \vrule height6pt  width6pt depth0pt}
\newcommand{\norm}[1]{ \| #1  \|}
\newcommand{\bnorm}[1]{ \big\| #1  \big\|}
\newcommand{\Bnorm}[1]{ \Big\| #1  \Big\|}

\newcommand{\co}{\colon}

\newcommand{\ot}{\otimes}
\newcommand{\ovl}{\overline}
\newcommand{\otvn}{\ovl\ot}
\newcommand{\ul}{\mathcal{U}}
\newcommand{\D}{\mathrm{D}}

\let\i\relax %\i ne fait plus rien !
\newcommand{\i}{\mathrm{i}}
\newcommand{\ov}{\overset}
\newcommand{\sa}{\mathrm{sa}}
\newcommand{\JW}{\mathrm{JW}}
\newcommand{\w}{\mathrm{w}} %weak

\renewcommand{\d}{\mathop{}\mathopen{}\mathrm{d}} %opérateur différentiel
\newcommand{\e}{\mathrm{e}} %constante e
\renewcommand{\d}{\mathop{}\mathopen{}\mathrm{d}} 
\let\cal\relax
\newcommand{\cal}{\mathcal}
\DeclareMathOperator{\tr}{Tr} %trace
\DeclareMathOperator{\Ran}{Ran} %range
\DeclareMathOperator{\dom}{dom} %domaine
\newtheorem{thm}{Theorem}[section]

\newtheorem{prop}[thm]{Proposition}

\newtheorem{lemma}[thm]{Lemma}

\newtheorem{remark}[thm]{Remark}

\newenvironment{proof}[1][]{\noindent {\it Proof #1} : }{\hbox{~}\qed
\smallskip
}

\numberwithin{equation}{section}
\usepackage[nottoc,notlot,notlof]{tocbibind}%biblio dans sommaire
\usepackage{tocloft}
\let\OLDthebibliography\thebibliography
\renewcommand\thebibliography[1]{
  \OLDthebibliography{#1}
  \setlength{\parskip}{0pt}
  \setlength{\itemsep}{0pt plus 0.3ex}
}
\begin{document}
\selectlanguage{english}
\title{\bfseries{2-positive contractive projections on noncommutative $\L^p$-spaces}}
\author{\bfseries{C\'edric Arhancet - Yves Raynaud}}
\date{}

\maketitle

%%%%%%%%%%%%%%%%%%%%%%%%%%%%%%%%%%%%%%%%%%%%%%%%%%%%%%%%%%%%%%%%
%%%%%%%%%%%%%%%%%%%%%%%%%%%%%%%%%%%%%%%%%%%%%%%%%%%%%%%%%%%%%%%%
\begin{abstract}
We prove the first theorem on projections on general noncommutative $\L^p$-spaces associated with non-type I von Neumann algebras where $1 \leq p < \infty$. This is the first progress on this topic since the seminal work of Arazy and Friedman [Memoirs AMS 1992] where the problem of the description of contractively complemented subspaces of noncommutative $\L^p$-spaces is explicitly raised. We show that the range of a 2-positive contractive projection on an arbitrary noncommutative $\L^p$-space is completely order isometrically isomorphic to some noncommutative $\L^p$-space. This result is sharp and is even new for Schatten spaces $S^p$. Our approach relies on non-tracial Haagerup's noncommutative $\L^p$-spaces in an essential way, even in the case of a projection acting on a Schatten space and is unrelated to the methods of Arazy and Friedman. 
\end{abstract}

%%%%%%%%%%%%%%%%%%%%%%%%%%%%%%%%%%%%%%%%%%%%%%%%%%%%%%%%%%%%%%%%
%%%%%%%%%%%%%%%%%%%%%%%%%%%%%%%%%%%%%%%%%%%%%%%%%%%%%%%%%%%%%%%%

\makeatletter
 \renewcommand{\@makefntext}[1]{#1}
 \makeatother
 \footnotetext{\\%\today
%This work is partially supported by ??.\\
%ANR Project OSQPI (ANR-11-BS01-0008)2010 
\noindent {\it 2020 Mathematics subject classification:}
 Primary 46L51, 46L07. %  46L51  Noncommutative measure and integration.%primary Secondary, 46L51 
%  46M35       Abstract interpolation of topological vector spaces [See also 46B70]
% 46L07       Operator spaces and completely bounded maps [See also 47L25]
% 43A22 Homomorphisms and multipliers of function spaces on groups, semigroups, etc.
% 43A15 Lp-spaces and other function spaces on groups, semigroups, etc.
\\
{\it Key words and phrases}: noncommutative $\L^p$-spaces, projections, complemented subspaces, conditional expectations.}%noncommutative $L^p$-spaces, Fourier
%multipliers, Schur multipliers,

{
  \hypersetup{linkcolor=blue}
\tableofcontents
}

%%%%%%%%%%%%%%%%%%%%%%%%%%%%%%%%%%%%%%%%%%%%%%%%%%%%%%%%%%%%%%%%%%%%%%%%%%%%%%%%%%%%%%%%%%%%%%%%%%%%%%%%%%%%%%%%%%%%%%%%%%%%%%%%%%%%%%%%%%%%%%%%%%%%%%
\section{Introduction}
%%%%%%%%%%%%%%%%%%%%%%%%%%%%%%%%%%%%%%%%%%%%%%%%%%%%%%%%%%%%%%%%%%%%%%%%%%%%%%%%%%%%%%%%%%%%%%%%%%%%%%%%%%%%%%%%%%%%%%%%%%%%%%%%%%%%%%%%%%%%%%%%%%%%%%

The study of projections and complemented subspaces has been at the heart of the study of Banach spaces since the inception of the field, see \cite{Rand01} and \cite{Mos06} for surveys. Recall that a projection $P$ on a Banach space $X$ is a bounded operator $P \co X \to X$ such that $P^2=P$ and that a complemented subspace $Y$ of $X$ is the range of a bounded linear projection $P$. If the projection is contractive, we say that $Y$ is contractively complemented.

Suppose that $1 \leq p <\infty$. A classical result from seventies essentially due to Ando \cite{And66}, and completed by \cite{Tza69} and \cite{BeL74}  tells that a subspace $Y$ of a classical (=commutative) $\L^p$-space $\L^p(\Omega)$ is contractively complemented if and only if $Y$ is isometrically isomorphic to an $\L^p$-space $\L^p(\Omega')$ (see also \cite{Dou65}, \cite{HoT09}, \cite{Ray04}, \cite{See66}). Moreover, $Y$ is the range of a positive contractive projection if and only if there exists an isometrical order isomorphism from $Y$ onto some $\L^p$-space $\L^p(\Omega')$, see \cite[Theorem 4.10]{Rand01} and \cite[Exercise 1 p.~235]{AbA02}. A positive contractive projection on $\L^p(\Omega)$, $1 \leq p < \infty$, $p\ne 2$, acts on the band generated by its range as a weighted conditional expectation \cite{BeL74}. This last result was extended later to a larger class of Banach function spaces \cite{DHdP90}.

It is natural to examine the case of noncommutative $\L^p$-spaces associated to von Neumann algebras. Schatten spaces are the most basic examples of noncommutative $\L^p$-spaces, these are the spaces $S^p$ of all operators $x \co \ell^2 \to \ell^2$ such that $\norm{x}_p=\bigl(\tr(\vert x\vert^p)\bigr)^{\frac{1}{p}}$ is finite. It is known from a long time that the range of a contractive projection $P \co S^p \to S^p$ on a Schatten space $S^p$ is not necessarily isometric to a Schatten space. It is a striking difference with the world of commutative $\L^p$-spaces of measure spaces. Indeed, in their remarkable memoirs \cite{ArF78} and \cite{ArF92}, Arazy and Friedman have succeeded in establishing a complete classification of contractively complemented subspaces of $S^p$. Building blocks of contractively complemented subspaces of $S^p$ are the $\L^p$-norm adaptations of the so-called Cartan factors of type I-IV. 

The description of general contractively complemented subspaces of noncommutative $\L^p$-spaces is an open problem raised explicitly in \cite[p.~99]{ArF92}. If $p=1$, Friedman and Russo \cite{FrB85} have given a description of the ranges of contractive projections on preduals (=noncommutative $\L^1$-spaces) of von Neumann algebras. Such a subspace is isometric to the predual of a $\mathrm{JW}^*$-triple, that is a weak* closed subspace of the space $\B(H,K)$ of bounded operators between Hilbert spaces $H$ and $K$ which is closed under the triple product $xy^*z+zy^*x$. Actually, the Friedman-Russo result is valid for projections acting on the predual of a $\mathrm{JW}^*$-triple, not just on the predual of a von Neumann algebra. 

Since Pisier's work \cite{Pis98} \cite{Pis03}, we can consider noncommutative $\L^p$-spaces and their complemented subspaces in the framework of operator spaces and completely bounded maps \cite{EfR00}. Using Arazy-Friedman Theorem, Le Merdy, Ricard and Roydor \cite[Theorem 1.1]{LRR09} characterized the completely 1-complemented subspaces of $S^p$. They turn out to be the direct sums of spaces of the form $S^p(H,K)$, where $H$ and $K$ are Hilbert spaces. The strategy of their proof is to examine individually each case provided by Arazy-Friedman Theorem and to determine if the associated contractive projection is completely contractive. See also \cite{NO02}, \cite{NeR11} for related results.

%Banach space structure of non-commutative $L^p$-space

Let $\cal{M}$ be a von Neumann algebra and assume that $1 \leq p <\infty$. Recall that a linear map $T \co \L^p(\cal{M}) \to \L^p(\cal{M})$ is $n$-positive for some integer $n \geq 1$ if the tensor product $\Id_{S^p_n} \ot T \co S^p_n(\L^p(\cal{M})) \to S^p_n(\L^p(\cal{M}))$ is a positive map. In particular, such a map is positive. Our main result is the following theorem which implies that the range of a 2-positive contractive projection $P \co \L^p(\cal{M}) \to \L^p(\cal{M})$ on some noncommutative $\L^p$-space $\L^p(\cal{M})$ is completely order isometrically isomorphic to some noncommutative $\L^p$-space $\L^p(\cal{N})$ for some von Neumann algebra $\cal{N}$.

%(see Theorem \ref{thm-ccp-proj-Lp} for a more complete statement).

\begin{thm}[Main Theorem]
\label{main-th-ds-intro}
Consider a von Neumann algebra $\cal{M}$. Suppose that $1 \leq p < \infty$. Let $P \co \L^p(\cal{M}) \to \L^p(\cal{M})$ be a 2-positive contractive projection. 
\begin{enumerate}
	\item There exist a projection $s(P) \in \cal{M}$, a von Neumann subalgebra $\cal{N}$ of $s(P)\cal{M}s(P)$, a normal semifinite faithful weight $\psi$ on $s(P)\cal{M}s(P)$ and a normal faithful conditional expectation $\E \co s(P)\cal{M}s(P) \to s(P)\cal{M}s(P)$ with range $\cal{N}$ leaving $\psi$ invariant such that 
\begin{enumerate}

\item $P(\L^p(\cal{M})) = s(P)P(\L^p(\cal{M}))s(P) = P(s(P)\L^p(\cal{M})s(P))$,

\item the restriction of the projection $P$ to $\L^p(s(P)\cal{M}s(P))$ identifies with the $\L^p$-extension $\E_{p} \co \L^p(s(P)\cal{M}s(P)) \to \L^p(s(P)\cal{M}s(P))$ naturally associated with $\E$ and the weight $\psi$,

\item the projection $P$ is $\cal{N}$-bimodular on $s(P)\L^p(\cal{M})s(P)$, that is
$$
P(x k y) 
= xP(k)y, \quad  x,y \in \cal{N}, k \in s(P)\L^p(\cal{M})s(P),
$$
In particular, the range $P(\L^p(\cal{M}))$ is a $\cal{N}$-bimodule.

\item the range $P(\L^p(\cal{M}))$ is completely order $\cal{N}$-bimodular isometrically isomorphic to $\L^p(\cal{N})$. 
\end{enumerate}
	
\item Moreover, the $\sigma$-finite projections in $\cal{N}$ are exactly the (left or right) supports of the elements of $P(\L^p(\cal{M}))$.  
	
\item Finally, if $1 < p < \infty$ then for any $x \in \L^p(\cal{M})$ we have $P(x)=P(s(P)xs(P))$ and the projection $P$ is necessarily completely positive.
\end{enumerate}
\end{thm}

The $\L^p$-extension $\E_{p}$ on noncommutative $\L^p$ spaces appear in the literature on noncommutative probability, in association with a normal {\it state} $\psi$, see e.g.~the works \cite{Gol85}, \cite{HiT87}, \cite{JuX03}, \cite{HJX10} and \cite{AcC82}. The most achieved one is probably that of \cite{JuX03}.
%The majority of the earlier sources use interpolation theory, leading to apparently different kinds of conditional expectations. These works were subsumed and unified by the paper \cite{JuX03} where the interpolation arguments were replaced by a smart handling of Haagerup's construction. 
We cannot limit ourselves to this <<probabilistic frame>>, since the von Neumann algebras that we are dealing with are not necessarily countably decomposable (=$\sigma$-finite), and thus may not support a faithful normal state. For this reason, we give a definition of these maps which differs from those of these authors, but follows a very classical scheme in commutative probability, and then we check  the coherence of this new definition with that of \cite{JuX03} in the <<probabilistic case>>. 
%Following a very classical scheme in commutative probability, this definition is simply by dualizing the inclusion map of the dual $\L^p$-space of the small algebra into the dual $\L^p$ of the big algebra (that inclusion results from the $\psi$-invariance condition for $\E$ by the theory of Haagerup $\L^p$-spaces, as shown in \cite{JuX03} in the probabilistic case). The dualization idea appears already in \cite{HiT87}, but the interpolation point of view adopted there is not convenient, as pointed above.

The result stated in the main theorem is even new for Schatten spaces $S^p$. The assumption of 2-positivity cannot be dropped. Indeed, if $\sigma \co S^p \to S^p$, $[x_{ij}] \mapsto [x_{ji}]$ denotes the transpose map then the linear map $P \ov{\mathrm{def}}{=}\frac{1}{2}(\Id_{S^p} +\sigma) \co S^p \to S^p$ is a positive contractive projection on the space $\{x \in S^p\,:\, \sigma(x)=x\}$ of symmetric matrices. 

The more general and more complicated case of \textit{positive} contractive projections will be investigated in a companion paper \cite{Arh20}, where we will make the link with the theory of $\JW^*$ algebras and their complex interpolation spaces and where we will describe the \textit{fine} structure of positive contractive projections. In \cite{Arh23b}, we will introduce nonassociative $\L^p$-spaces associated to $\JW^*$-algebras by interpolation and will show that all tracial nonassociative $\L^p$-spaces from $\JW^*$-factors arise as positively contractively complemented subspaces of noncommutative $\L^p$-spaces. In \cite{Arh23a}, we will describe the structure of contractively decomposable projections. Note that the class of decomposable maps is thoroughly studied in the memoir \cite{ArK23}. Finally, we refer to \cite{PiX03} and \cite{HRS03} for more information on the structure of noncommutative $\L^p$-spaces. 

\paragraph{Approach of the paper} 
We consider the <<support projection>> $s(P)\ov{\mathrm{def}}{=} \bigvee_{h \in \Ran P, h \geq 0} s(h)$ of the non-zero 2-positive contractive projection $P \co \L^p(\cal{M}) \to \L^p(\cal{M})$. Here $s(h)$ is the support projection of the element $h$. It is easy to check the equality $
\Ran P
=P(s(P)\L^p(\cal{M})s(P))$. 

For the $\sigma$-finite case, we establish in Proposition \ref{Prop-sh=sP} that there exists a non-zero positive element $h$ of the space $\Ran P$ such that $s(P)=s(h)$, crucially using the assumption of $\sigma$-finiteness. Hence we have the equality
\begin{equation}
\label{intro-fin}
\Ran P
=P(s(h)\L^p(\cal{M})s(h)). 
\end{equation}
We continue this presentation by assuming that the noncommutative $\L^p$-space $\L^p(\cal{M})$ is defined by a normal faithful state $\varphi$ of the von Neumann algebra $\cal{M}$. We will see that we can suppose that the support projection $s(h)$ belongs to the centralizer of the state $\varphi$.

Our approach relies on a symmetric two-sided lifting result (Theorem \ref{Th-relevement-cp}) allowing us to introduce a unital normal 2-positive map $\E \co s(h)\cal{M}s(h) \to s(h)\cal{M}s(h)$ satisfying
\begin{equation*}
\quad P\big(h^{\frac{1}{2}}xh^{\frac{1}{2}}\big)
=h^{\frac{1}{2}} \E(x) h^{\frac{1}{2}}, \qquad x \in s(h)\cal{M}s(h).
\end{equation*}
This reduces the problem to the analysis of this map.

In Lemma \ref{Lemma-faithful}, we subtly show that this map $\E$ is a normal faithful conditional expectation onto a von Neumann subalgebra $\cal{N}$ of the reduced von Neumann algebra $\cal{M}_h \ov{\mathrm{def}}{=} s(h)\cal{M}s(h)$, using geometric properties of noncommutative $\L^p$-spaces. Afterwards, we introduce in \eqref{psih-def} a normal faithful positive linear form $\psi$ on the reduced von Neumann algebra $\cal{M}_h$ defined by
\begin{equation*}
\psi(x)
\ov{\mathrm{def}}{=} \tr_\varphi(h^px), \quad x \in \cal{M}_h
\end{equation*}
where $\tr_\varphi$ is the <<Haagerup trace>> associated with the state $\varphi$. We also consider the restriction $\varphi_{s(h)}$ of the state $\varphi$ to the reduced von Neumann algebra $\cal{M}_h$.

We then show in Lemma \ref{Lemma-prservation-state} that the conditional expectation $\E$ preserves $\psi$.
Consequently, we can consider the $\L^p$-extension $\E_p \co \L^p(\cal{M}_h,\psi) \to \L^p(\cal{M}_h,\psi)$
of the conditional expectation $\E$. We also introduce the complete order and isometric <<change of density>> map $\kappa \co \L^p(\cal{M}_h,\psi) \to \L^p(\cal{M}_h,\varphi_{s(h)})$. We then prove in \eqref{sans-fin-3} that the restriction of the projection $P$ to the subspace $s(h)\L^p(\cal{M})s(h)$ can be expressed as
\begin{equation}
\label{sans-fin-intro1}
P|_{s(h)\L^p(\cal{M})s(h)}
=\kappa \circ \E_p \circ \kappa^{-1}
\end{equation}
where we use the classical identification of the subspace $s(h)\L^p(\cal{M})s(h)$ with the noncommutative $\L^p$-space $\L^p(\cal{M}_h,\varphi_{s(h)})$, which is justified since the projection $s(h)$ belongs to the centralizer of the state $\varphi$. In particular, we obtain that the range
$$
\Ran P 
\ov{\eqref{intro-fin}}{=} P(s(h)\L^p(\cal{M})s(h))
\ov{\eqref{sans-fin-intro1}}{=} \kappa\E_p\kappa^{-1}(\L^p(\cal{M}_h)).
$$
is complete order isometrically isomorphic to a noncommutative $\L^p$-space.

For the non-$\sigma$-finite case, we consider the set $\mathcal{P}(P)$ of all support projections of positive elements in the range $\Ran P$ of the projection $P$. We demonstrate in Theorem \ref{Thm-Disjoint} that there exists a family $(s_i)_{i \in I}$ of pairwise disjoint projections in the set $\mathcal{P}(P)$ such that $
s(P)
=\bigvee_{i \in I} s_i$. For each $i \in I$, we introduce a positive element $h_i$ of $\Ran P$ with support projection $s(h_i)=s_i$. Subsequently, we initiate a gluing procedure that allows us to essentially use the $\sigma$-finite case.%Let $\psi_i\ov{\mathrm{def}}{=} \tr(h_i^p\, \cdot)$ be the normal positive linear form on the von Neumann algebra $\cal{M}$ associated with the positive element $h_i^p$ of $\L^1(\cal{M},\varphi)$.

It should also be noted that the use of \textit{non-tracial} Haagerup's noncommutative $\L^p$-spaces is crucial even for the case of \textit{tracial} noncommutative $\L^p$-spaces.

\paragraph{Structure of the paper} The paper is organized as follows. Section \ref{Haagerup-noncommutative} gives a presentation of Haagerup's noncommutative $\L^p$-spaces, including some complementary results. In the next Section \ref{Sec-conditional}, we investigate $\L^p$-extensions of normal conditional expectations on noncommutative $\L^p$-spaces in a general formulation which makes sense for non-$\sigma$-finite von Neumann algebras. These results are specially used in Section \ref{End-of-the-end}. In Section \ref{Sec-projections-von}, a folklore crucial fact on positive projections acting on von Neumann algebras is highlighted. In Section \ref{Positive-maps}, we describe our lifting result of positive maps acting on noncommutative $\L^p$-spaces. In Section \ref{From-local-to-global}, we define several notions of support projections of a contractive projection. Section \ref{Sec-a-local} contains a local description of the action of a 2-positive contractive projection. We deduce a proof of Theorem \ref{main-th-ds-intro} in the $\sigma$-finite case. In Section \ref{Carac-Nh} and Section \ref{End-of-the-end}, we present the more elaborated treatment of the general case in the non-sigma-finite case. %At this stage the two first points in our main theorem are proven. The last point of the Main Theorem is a consequence of the general result presented in the last .

%%%%%%%%%%%%%%%%%%%%%%%%%%%%%%%%%%%%%%%%%%%%%%%%%%%%%%%%%%%%%%%%%%%%%%%%%%%%%%%%%%%%%%%%%%%%%%%%%%%%%%%%%%%%%%%%%%%%%%%%%%%%%%%%%%%%%%%%%%%%%%%%%%%%%%%
%\section{Preliminaries}
%\label{sec-Preliminaries}
%%%%%%%%%%%%%%%%%%%%%%%%%%%%%%%%%%%%%%%%%%%%%%%%%%%%%%%%%%%%%%%%%%%%%%%%%%%%%%%%%%%%%%%%%%%%%%%%%%%%%%%%%%%%%%%%%%%%%%%%%%%%%%%%%%%%%%%%%%%%%%%%%%%%%%%

%%%%%%%%%%%%%%%%%%%%%%%%%%%%%%%%%%%%%%%%%%%%%%%%%%%%%%%%%%%%%%%%%%%%%%%%%%%%%%%%%%%%%
\section{Haagerup's noncommutative $\L^p$-spaces}
\label{Haagerup-noncommutative}

The readers are referred to the surveys \cite{Kos14}, \cite{PiX03}, \cite{Terp81}, \cite{Ray03} and references therein for information on noncommutative $\L^p$-spaces.

It is well-known that there exist several equivalent constructions of noncommutative $\L^p$-spaces associated with a von Neumann algebra. In this paper, we will use Haagerup's noncommutative $\L^p$-spaces introduced in \cite{Haa79b} and presented in a more detailed way in \cite{Terp81}. Here, $\cal{M}$ will denote a general von Neumann algebra acting on a Hilbert space $H$ and we denote by $s_l(x)$ and $s_r(x)$ the left support and the right support of an operator $x$. If $x$ is a positive operator then $s_l(x)=s_r(x)$ is called the support of $x$ and denoted by $s(x)$.

If $\cal{M}$ is equipped with a normal semifinite faithful trace, then the topological $*$-algebra of all (unbounded) $\tau$-measurable operators $x$ affiliated with $\cal{M}$ is denoted by $\L^0(\cal{M},\tau)$. If $a,b \in \L^0(\cal{M},\tau)_+$, we have
\begin{equation}
\label{Ine-measurable}
a \leq b
\iff \dom b^{\frac{1}{2}} \subset \dom a^{\frac{1}{2}}
\text{ and } \norm{a^{\frac{1}{2}}\xi}_H \leq \norm{b^{\frac{1}{2}}\xi}_H, \text{ for any } \xi \in \dom b^{\frac{1}{2}}.
\end{equation}
If $a,b \in \L^0(\cal{M},\tau)$, we have
\begin{equation}
\label{support-properties}
ab=0 \Rightarrow s_r(a)b=0 \text{ and } as_l(b)=0.
\end{equation}

In the sequel, we fix a normal semifinite faithful weight $\varphi$ on $\cal{M}$ and $\sigma^\varphi=(\sigma_t^\varphi)_{t \in \R}$ denote the one-parameter modular automorphisms group associated with $\varphi$ \cite[p.~92]{Tak03}. 

%Let $h \in M_+$, and consider the map $\varphi_h \co M_+ \to [0,\infty]$ given by $\varphi_h(x) = \varphi(h^{\frac{1}{2}}xh^{\frac{1}{2}})$. This map is a normal weight. 
%%(Boey p49)
%Furthermore if $h$ is invertible, then $\varphi_h$ is faithful and semifinite. More generally, with an affiliated operator $h$ we can define a weight $\varphi_h$ \cite[page 103]{Tak03}. If $M$ is a von Neumann algebra equipped with a normal semifinite faithful trace $\tau$ then every normal semi-finite weight $\varphi$ can be written uniquely in the form $\varphi=\tau_h$ with a positive selfadjoint operator $h$ affiliated with $M$ \cite[page 122]{Tak03} (often denoted by $\tau(h\cdot)$).%and called Radon-Nikodym derivative

For $1 \leq p < \infty$, the Banach spaces $\L^p(\cal{M})$ are constructed as spaces of measurable operators relatively not to $\cal{M}$ but to some semifinite bigger von Neumann algebra, namely, the crossed product $\tilde{\cal{M}} \ov{\mathrm{def}}{=} \cal{M} \rtimes_{\sigma^\varphi} \R$ of $\cal{M}$ the modular automorphisms group, that is, the von Neumann subalgebra of $\B(\L^2(\R,H))$ generated by the operators $\pi(x)$ and $\lambda_s$, where $x \in \cal{M}$ and $s \in \R$, defined by
$$
\big(\pi(x)\xi\big)(t) \ov{\mathrm{def}}{=} \sigma^{\varphi}_{-t}(x)(\xi(t))
\quad \text{and} \quad
\lambda_s(\xi(t)) \ov{\mathrm{def}}{=} \xi(t-s), \quad t \in \R, \ \xi \in \L^2(\R,H).
$$ 
For any $s \in \R$, let $W(s)$ be the unitary operator on $\L^2(\R,H)$ defined by
$$
\big(W(s)\xi\big)(t)
\ov{\mathrm{def}}{=} \e^{-\i s t} \xi(t),\quad \xi \in \L^2(\R,H).
$$
The dual action $\hat{\sigma} \co \R \to \B(\tilde{\cal{M}})$ on the von Neumann algebra $\tilde{\cal{M}}$ \cite[p.~260]{Tak03} is given by
\begin{equation}
\label{Dual-action}
\widehat{\sigma}_s(x)
\ov{\mathrm{def}}{=} W(s)xW(s)^*,\quad x \in \tilde{\cal{M}},\ s \in \R.
\end{equation}
Then, by \cite[p.~259]{Tak03}%\cite[Lemma 3.6]{Haa4}
, $\pi(\cal{M})$ is the fixed subalgebra of $\tilde{\cal{M}}$ under the family of automorphisms $\widehat{\sigma}_s$:
\begin{equation}
\label{carac-Pi-de-M}
\pi(\cal{M})
=\big\{x \in \tilde{\cal{M}}\ : \ \widehat{\sigma}_s(x)=x \quad \text{for all }s \in \R\big\}.
\end{equation}
We identify $\cal{M}$ with the subalgebra $\pi(\cal{M})$ in the crossed product $\tilde{\cal{M}}$. If $\psi$ is a normal semifinite weight on $\cal{M}$, we denote by $\widehat{\psi}$ its Takesaki's dual weight on the crossed product $\tilde{\cal{M}}$. We can give the following definition of \cite{Haa78}. %defined in \cite[page 300]{Str81}. 
Indeed, Haagerup introduces an operator valued weight $T \co \cal{M}^+ \to \bar{\cal{M}}^+$ with values in the extended positive part $\bar{\cal{M}}^+$ of $\cal{M}$, formally defined by 
\begin{equation}
\label{Operator-valued}
T(x)
=\int_\R \widehat{\sigma}_s(x)\d s
\end{equation}
and shows that for a normal semifinite weight $\psi$ on $\cal{M}$, its dual weight is 
\begin{equation}
\label{Def-poids-dual}
\widehat{\psi} \ov{\mathrm{def}}{=} \bar \psi\circ T
\end{equation}
where $\bar\psi$ denotes the natural extension of the normal weight $\psi$ to $\bar{\cal{M}}^+$. This dual weight  satisfies the $\widehat{\sigma}$-invariance relation $\widehat{\psi} \circ \widehat{\sigma}=\widehat{\psi}$, see \cite[(10) p.~26]{Terp81}. By \cite[Chap.~II, Lemma 1]{Terp81}%\cite[p.~301]{Str81} \cite[Theorem 3.7]{Haa4}
, the map $\psi \to \widehat{\psi}$ is a bijection from the set of normal semifinite weights on $\cal{M}$ onto the set of normal semifinite $\widehat{\sigma}$-invariant weights on $\tilde{\cal{M}}$.

Recall that by \cite[Lemma 5.2 and Remark p.~343]{Haa79a} and \cite[Theorem 1.1 (c)]{Haa78} the crossed product $\tilde{\cal{M}}$ is semifinite and there is a unique normal semifinite faithful trace $\tau=\tau_{\varphi}$ on $\tilde{\cal{M}}$ satisfying $(\mathrm{D}\widehat{\varphi}:\mathrm{D} \tau)_t=\lambda_t$ for any $t \in \R$ where $(\mathrm{D}\widehat{\varphi}:\mathrm{D} \tau)_t$ denotes the Radon-Nikodym cocycle \cite[p.~48]{Str81} \cite[p.~111]{Tak03} of the dual weight $\widehat{\varphi}$ with respect to $\tau$. Moreover, $\tau$ satisfies the relative invariance $\tau \circ \widehat{\sigma}_s = \e^{-s}\tau$ for any $s \in \R$ by \cite[Lemma 5.2]{Haa79a}.

If $\psi$ is a normal semifinite weight on the von Neumann algebra $\cal{M}$, we denote by $h_\psi$ the Radon-Nikodym derivative of the dual weight $\widehat{\psi}$ with respect to $\tau$ given by \cite[Theorem 4.10 p.~67]{Str81}. By \cite[Corollary 4.8 p.~67]{Str81}, note that the relation of $h_\psi$ with the Radon-Nikodym cocycle of $\widehat{\psi}$ is
\begin{equation}
\label{Radon-Nikodym-1}
(\D\widehat\psi:\mathrm{D}\tau)_t
= h_\psi^{\i t}, \quad t \in \R.
\end{equation}
If $\psi=\varphi$, we let $D_\varphi \ov{\mathrm{def}}{=} h_\varphi$ and we call it the density operator of $\varphi$.

By \cite[Chap.~II, Proposition 4]{Terp81}, the mapping $\psi \to h_\psi$ gives a bijective correspondence between the set of all normal semifinite weights on the von Neumann algebra $\cal{M}$ and the set of positive selfadjoint operators $h$ affiliated with $\tilde{\cal{M}}$ satisfying
\begin{equation}
\label{eq:def-L1}
\widehat{\sigma}_s(h)
=\e^{-s}h, \quad s \in \R.
\end{equation}
Moreover, by \cite[Chap.~II, Corollary 6]{Terp81}, $\omega$ belongs to $\cal{M}_*^+$ if and only if $h_\omega$ belongs to $\L^0(\tilde{\cal{M}},\tau)_+$. One may extend by linearity the map $\omega \mapsto h_\omega$ to the whole of $\cal{M}_*$. The Haagerup space $\L^1(\cal{M},\varphi)$ is defined as the set $\{h_\omega : \omega \in \cal{M}_*\}$, i.e.~the range of the previous map. %This is a closed linear subspace of $\L^0(\tilde{\cal{M}},\tau)$, characterized by the conditions \eqref{eq:def-L1}.

By \cite[Chap.~II, Theorem 7]{Terp81}, the mapping $\omega \mapsto h_\omega$, $\cal{M}_* \to \L^1(\cal{M},\varphi)$ is a linear order isomorphism which preserves the conjugation, the module, and the left and right actions of $\cal{M}$. Then $\L^1(\cal{M},\varphi)$ may be equipped with a continuous linear functional $\tr \co \L^1(\cal{M}) \to \C$  defined by
\begin{equation}
\label{Def-tr}
\tr h_\omega
\ov{\mathrm{def}}{=} \omega(1),\quad \omega \in \cal{M}_*
\end{equation}
\cite[Chap.~II, Definition 13]{Terp81}. A norm on the space $\L^1(\cal{M},\varphi)$ may be defined by $\norm{h}_1 \ov{\mathrm{def}}{=} \tr(|h|)$ for every $h \in \L^1(\cal{M},\varphi)$. By \cite[Chap.~II, Proposition 15]{Terp81}, the map $\cal{M}_* \to \L^1(\cal{M},\varphi)$, $\omega \mapsto h_\omega$ is a surjective isometry.

More generally for $1 \leq p \leq \infty$, the Haagerup $\L^p$-space $\L^p(\cal{M},\varphi)$ associated with the normal faithful semifinite weight $\varphi$ is defined \cite[Chap.~II, Definition 9]{Terp81} as the subset of the topological $*$-algebra $\L^0(\tilde{\cal{M}},\tau)$ of all (unbounded) $\tau$-measurable operators $x$ affiliated with $\tilde{\cal{M}}$ satisfying for any $s \in \R$ the condition
\begin{equation}
\label{Def-Haagerup}
\widehat{\sigma}_s(x)
=\e^{-\frac{s}{p}}x \quad \text{if } p<\infty 
\quad \text{and} \quad 
\widehat{\sigma}_s(x)
=x \quad \text{ if } p=\infty	
\end{equation}
where $\widehat{\sigma}_s \co \L^0(\tilde{\cal{M}},\tau) \to \L^0(\tilde{\cal{M}},\tau)$ is here the continuous $*$-automorphism obtained by a natural extension of the dual action \eqref{Dual-action} on $\mathcal{M}$. By \eqref{carac-Pi-de-M}, the space $\L^\infty(\cal{M},\varphi)$ coincides with $\pi(\cal{M})$ that we identify with $\cal{M}$. The spaces $\L^p(\cal{M},\varphi)$ are closed selfadjoint linear subspaces of the space $\L^0(\tilde{\cal{M}},\tau)$. They are closed under left and right multiplications by elements of $\cal{M}$. If $h=u|h|$ is the polar decomposition of $h \in \L^0(\tilde{\cal{M}},\tau)$ then by \cite[Chap.~II, Proposition 12]{Terp81} we have
$$
h \in \L^p(\cal{M},\varphi)
\iff 
u \in \cal{M} \text{ and } |h| \in \L^p(\cal{M},\varphi).
$$
%By \cite[Chap.~II, Prop. 10]{Terp81},

Suppose that $1 \leq p<\infty$. By \cite[Chap.~II, Proposition 12]{Terp81} and its proof, for any $h \in \L^0(\tilde{\cal{M}},\tau)_+$, we have $h^p \in \L^0(\tilde{\cal{M}},\tau)_+$. Moreover, an element $h \in \L^0(\tilde{\cal{M}},\tau)$ belongs to the space $\L^p(\cal{M},\varphi)$ if and only if $|h|^p$ belongs to the space $\L^1(\cal{M},\varphi)$. A norm on the vector space $\L^p(\cal{M},\varphi)$ is then defined by the formula 
\begin{equation}
\label{Def-norm-Lp}
\norm{h}_p
\ov{\mathrm{def}}{=} (\tr |h|^p)^{\frac{1}{p}}
\end{equation} 
if $1 \leq p < \infty$ and by $\norm{h}_\infty \ov{\mathrm{def}}{=}\norm{h}_\cal{M}$, see \cite[Chap.~II, Definition 14]{Terp81}.

\paragraph{Duality} Let $p, p^* \in [1,\infty]$ with $\frac{1}{p}+\frac{1}{p^*}=1$. By \cite[Chap.~II, Proposition 21]{Terp81}, for any $h \in \L^p(\cal{M},\varphi)$ and any $k \in \L^{p^*}(\cal{M},\varphi)$ the elements $hk$ and $kh$ belong to the space $\L^1(\cal{M},\varphi)$ and we have the tracial property $\tr(hk)=\tr(kh)$.

If $1 \leq p<\infty$, by \cite[Ch.~II, Theorem 32]{Terp81} the bilinear form $\L^p(\cal{M},\varphi) \times \L^{p^*}(\cal{M},\varphi)\to \C$, $(h,k) \mapsto \tr(h k)$ defines a duality bracket between the Banach spaces $\L^p(\cal{M},\varphi)$ and $\L^{p^*}(\cal{M},\varphi)$, for which $\L^{p^*}(\cal{M},\varphi)$ is isometrically the dual of $\L^p(\cal{M},\varphi)$. 

On the other hand, if the weight $\varphi$ is tracial, i.e.~$\varphi(x^*x)=\varphi(xx^*)$ for all $x \in \cal{M}$, then the Haagerup space $\L^p(\cal{M},\varphi)$ isometrically coincides with Dixmier's classical tracial noncommutative $\L^p$-space, see \cite[p.~62]{Terp81}.

\paragraph{Change of weight} It is essentially proved in \cite[p.~59]{Terp81} %(see also \cite[Theorem 5.1]{Ray03})  
that $\L^p(\cal{M},\varphi)$ is independent of $\varphi$ up to an isometric isomorphism preserving the order and modular structure of $\L^p(\cal{M},\varphi)$, as well as the external products and Mazur maps. In fact given two normal semifinite faithful weights $\varphi_1, \varphi_2$ on the von Neumann algebra $\cal{M}$ there is a $*$-isomorphism $\kappa \co \tilde{\mathcal{M}}_1 \to \tilde{\mathcal{M}}_2$ between the crossed products $\tilde{\cal{M}}_i\ov{\mathrm{def}}{=} \cal{M} \rtimes_{\sigma^{\varphi_i}} \R$ preserving $\cal{M}$, as well as the dual actions and pushing the trace on $\tilde{\cal{M}}_1$ onto the trace on $\tilde{\cal{M}}_2$, that is 
\begin{align}
\label{kappa}
\pi_2 = \kappa\circ \pi_1,\quad \hat\sigma_2\circ\kappa 
= \kappa\circ\hat\sigma_1 
\quad \text{and} \quad
\tau_2 \circ \kappa=\tau_1.
\end{align}
Furthermore, $\kappa$ extends naturally to a topological $*$-isomorphism $\hat{\kappa} \co \L^0(\tilde{\cal{M}}_1,\tau_1) \to \L^0(\tilde{\cal{M}}_2,\tau_2)$ between the algebras of measurable operators.%, which restricts to isometric $*$-isomorphisms between the respective Banach spaces $\L^p(\cal{M}_1,\varphi_1)$ and $\L^p(\cal{M}_2,\varphi_2)$.

Moreover, it turns out also that for every normal semifinite faithful weight $\psi$ on the von Neumann algebra $\cal{M}$, the dual weights $\hat\psi_i$ corresponds through $\kappa$, that is $\hat{\psi}_2 \circ \kappa=\hat\psi_1$. It follows that if $\omega \in \cal{M}_*$ the corresponding Radon-Nikodym derivatives must verify $h_{\omega,2}=\hat{\kappa}(h_{\omega,1})$. In particular if $\omega \in \cal{M}_*^+$, we have
\begin{equation}
\tr_1 h_{\omega,1}
\ov{\eqref{Def-tr}}{=} \omega(1) 
\ov{\eqref{Def-tr}}{=} \tr_2 h_{\omega,2}
=\tr_2 \hat{\kappa}(h_{\omega,1}).
\end{equation}
Hence $\hat{\kappa} \co \L^1(\cal{M},\varphi_1) \to \L^1(\cal{M},\varphi_2)$ preserves the functionals $\tr$:
\begin{equation}
\label{Trace-preserving}
\tr_1
=\tr_2 \circ \hat{\kappa}.
\end{equation}
Since $\hat{\kappa}$ preserves the $p$-powers operations, i.e.~$\hat{\kappa}(h^p)=\big(\hat{\kappa}(h)\big)^p$ for any $h \in \L^0(\cal{M}_1)$, it induces an isometry from $\L^p(\cal{M},\varphi_1)$ onto $\L^p(\cal{M},\varphi_2)$, preserving the $\cal{M}$-bimodule structures. It is not hard to see that this isometry is a complete order isomorphism.%, a fact which is of first importance for our study.
 
This independence allows us to consider $\L^p(\cal{M},\varphi)$ as a particular realization of an abstract space $\L^p(\cal{M})$. The $\cal{M}$-bimodule structure and the norm of the space $\L^p(\cal{M})$ are defined unambiguously by those of any of its particular realization, as well as the trace functional of $\L^1(\cal{M})$ and the bilinear products $\L^p(\cal{M}) \times \L^q(\cal{M}) \to \L^r(\cal{M})$ where $\frac 1r=\frac 1p+\frac 1q$, and the Mazur maps $\L^p_+(\cal{M}) \to \L^1_+(\cal{M})$, $h \mapsto h^p$ (and their inverses). An element $h \in \L^1(\cal{M})$ identifies to the linear form $\psi \in \cal{M}_*$ defined by the condition $\psi(x)=\tr(xh)$ where $x \in \cal{M}$.%, and the positive part $\L^p_+(\cal{M})$ may be seen as the cone of $p$-roots $\psi^{\frac{1}{p}}$ of positive elements of $\cal{M}_*$.

%A key fact for our analysis of positive contractions on noncommutative $\L^p$-spaces is  that for every projection $e$ in $\cal{M}$ the linear subspace $e\L^p(\cal{M})e \ov{\mathrm{def}}{=} \{ehe : h \in \L^p(\cal{M})\}$ is completely order isometric to $\L^p(e\cal{M}e)$, the $\L^p$-space of the reduced von Neumann algebra $e\cal{M}e$. %This fact is not evident from a given realization of $\L^p(\cal{M},\varphi)$, since the restriction $\varphi_e$ of $\varphi$ to $e\cal{M}e$ may not be semifinite, or the crossed product $\R \rtimes_{\sigma^{\varphi_e}} e\cal{M}e$ not be a reduction of $\R \rtimes_{\sigma^{\varphi}} \cal{M}$.

\paragraph{Case of a normal faithful linear form} If $\varphi$ is a normal faithful linear form on the von Neumann algebra $\cal{M}$ then by \cite[(1.13)]{HJX10} the density operator $D_\varphi$ belongs to the space $\L^1(\cal{M},\varphi)$ and 
\begin{equation}
\label{HJX-1.13}
\varphi(x)
=\tr(D_\varphi x)
=\tr(xD_\varphi ), \quad x \in \cal{M}.
\end{equation}

\paragraph{Reduced noncommutative $\L^p$-spaces} Recall that the centralizer \cite[p.~38]{Str81} of a normal semifinite faithful weight $\varphi$ is the von Neumann subalgebra $\cal{M}^\varphi \ov{\mathrm{def}}{=}\{x \in \cal{M}: \sigma_t^\varphi(x)=x\text{ for all } t \in \R\}$. If $x \in \cal{M}$, we have by \cite[(2) p.~39]{Str81}
\begin{equation}
\label{Charac-centralizer}
x \in \cal{M}^\varphi
\iff x\mathfrak{m}_{\varphi} \subset \mathfrak{m}_{\varphi},\, \mathfrak{m}_{\varphi}x \subset \mathfrak{m}_{\varphi} \text{ and } \varphi(xy)=\varphi(yx) \text{ for any } y \in \mathfrak{m}_{\varphi}. 
\end{equation}
If $e$ belongs to the centralizer of $\varphi$, it is well-known that we can identify the space $\L^p(e\cal{M}e)$ with the subspace $e\L^p(\cal{M})e$ of the noncommutative $\L^p$-space $\L^p(\cal{M})$, see \cite[Lemma 4.3]{GoLa09} and \cite[p.~508]{Wat88}. Let us give some details.

%Let $\varphi$ be a faithful normal semifinite weight on a von Neumann algebra $\cal{M}$ on a
%Hilbert space $H$. 

For each projection $e$ in the centralizer $\cal{M}^\varphi$, let $\varphi_e$ be the restriction of $\varphi$ to the reduced von Neumann algebra $e\cal{M}e$. The weight $\varphi_e$ is still semifinite. From the KMS-condition, it is easy to see that $\sigma^{\varphi_e} = \sigma^\varphi |_{e\cal{M}e}$, and it follows that $e\cal{M}e \rtimes_{\sigma^{\varphi_e}} \R$ coincides with $\bar e(\cal{M} \rtimes_{\sigma^{\varphi}} \R) \bar e$, where $\bar e$ is the canonical image of $e$ in $\tilde{\cal{M}}$. Actually, we have $\bar e=e \ot \Id_{\B(\L^2(\R))}$. From \eqref{Operator-valued} and \eqref{Def-poids-dual} it is clear that  the dual weights to $\varphi$ and $\varphi_e$ are linked by the equation
$$
\widehat{\varphi_e} 
=(\widehat \varphi)_{\bar e}.
$$
Let  $\tau=\tau_\varphi$ be the canonical trace on the crossed product $\tilde{\cal{M}} = \cal{M} \rtimes_{\sigma^{\varphi}} \R $, and $h_\varphi= {\frac{\d}{\d \tau}}\widehat\varphi$. Note that $h_\varphi^{\i t}=\lambda_t$ commutes with $\bar e$, and  so does $h_\varphi^{\frac{1}{2}}$. Moreover, $eh_\varphi$ is a positive selfadjoint operator affiliated with $\bar e\cal{M} \bar e$. Then for any $x \in \bar e\tilde{\cal{M}}\bar e$, we have
$$
\widehat\varphi(x)
=\tau\big(h_\varphi^{\frac{1}{2}}xh_\varphi^{\frac{1}{2}}\big) 
= \tau\big(h_\varphi^{1/2}\bar ex\bar e h_\varphi^{\frac{1}{2}}\big)
= \tau_{\bar e}\big((\bar eh_\varphi)^{\frac{1}{2}}x(\bar e h_\varphi)^{\frac{1}{2}}\big)
$$
where $\tau_{\bar e}$ is the reduced (normal, semifinite) trace on the reduced von Neumann algebra $\bar e\tilde{\cal{M}}\bar e$. Hence by \cite[Theorem 4.10]{Str81}
$$
(\D\widehat{\varphi_e}:\D\tau_{\bar e})_t 
=(\bar e h_\varphi)^{\i t}=\bar e h_\varphi^{\i t}
= \bar e\lambda_t.
$$
Since $\bar e\lambda_t$ is the translation operator on $\bar e\tilde{\cal{M}}\bar e$ it follows by uniqueness that $\tau_{\bar e}$ coincides with the canonical trace on the crossed product $e\cal{M}e \rtimes_{\sigma^{\varphi_e}} \R$. Then it becomes clear that the $*$-algebra of $\tau_{\bar e}$-measurable operators on $\bar e \cal{M}\bar e$ is realized as an $\L^0$-closed $*$-subalgebra of $\L^0(\tilde{\cal{M}},\tau)$, namely $e\L^0(\tilde{\cal{M}},\tau)e$, either abstractly by completing the respective von Neumann algebras in their $\L^0$-topologies, or concretely by representing them as $*$-subalgebras of (unbounded) closed operators on the Hilbert spaces $\L^2(\R, eH)$ and $\L^2(\R, H)$ respectively if $\cal{M}$ is given as a von Neumann subalgebra of some $\B(H)$. Indeed, a (closed, unbounded) operator with dense domain $a$ on $\L^2(\R,eH)$ may be trivially extended to an operator $\tilde{a}$ on $\L^2(\R, H)$ by setting
$$
\tilde{a}\xi 
=ae\xi, \quad  \xi \in \dom \tilde a
=\dom a \oplus \L^2(\R,(1-e)H).
$$ 
Since the dual action of $\sigma^{\varphi_e}$ is the restriction of that of $\sigma^\varphi$ to the von Neumann algebra $e(\cal{M} \rtimes_{\sigma^{\varphi}} \R) e$, it is not difficult to show that the mapping $a \to ebe$ gives a topological $*$-isomorphism between $\L^0(\bar e\tilde{\cal{M}} \bar e,\tau_e)$ and $\bar e\L^0(\tilde{\cal{M}},\tau)\bar e$, sending the space $\L^p(e\cal{M}e)$ onto the space $e\L^p(\cal{M})e$ for every $1 \leq p \leq \infty$. 

If $\psi \in (e\cal{M}e)_*^+$ is a normal positive bounded linear form on the von Neumann algebra $e\cal{M}e$, we may consider its natural extension $i_e\psi$ to the von Neumann algebra $\cal{M}$ defined by $i_e\psi(x)=\psi(exe)$ where $x \in \cal{M}$. Using \eqref{Operator-valued} and \eqref{Def-poids-dual} it is clear that
$$
\widehat{i_e\psi\,}(x)
=\widehat \psi(\bar e x\bar e), \quad x \in \cal{M}.
$$
Since $h_\psi$ belongs to the space $\bar e \L^0(\cal{M},\tau)\bar e$, we have  
$$
\hat\psi(\bar e x\bar e) 
= \tau_{\bar e}\big(h_\psi^{\frac{1}{2}} \bar e x\bar e h_\psi^{\frac{1}{2}}\big) 
= \tau\big(\bar e (h_\psi^{\frac{1}{2}} \bar e x\bar e h_\psi^{\frac{1}{2}} \bar e)\big)
= \tau\big(h_\psi^{\frac{1}{2}}xh_\psi\big)^{\frac{1}{2}}.
$$
It is then clear by uniqueness of the Radon-Nikodym derivative that  $
h_{i_e\psi}
= h_\psi$. Since
$$
\tr_\varphi h_{i_e \psi}
= i_e\psi(1)
=\psi(e)
=\tr_{\varphi_e}h_\psi,
$$
we have the following lemma.

\begin{lemma}
\label{Lemme-trace coincide}
The Haagerup trace $\tr_\varphi$ restricts to $\tr_{\varphi_e}$  on $\L^1(e\cal{M}e)$. 
\end{lemma}
It follows at once that the inclusions $\L^p(e\cal{M}e)\subset \L^p(\cal{M})$ are isometric

\paragraph{Construction of suitable weight} Consider a projection $e$ of the von Neumann algebra $\cal{M}$. Let us construct a normal semifinite faithful weight on $\cal{M}$ with centralizer containing $e$. Choosing %with \cite[10.10]{Str81} \cite[Exercise 7.6.46]{KaR97} 
two normal semifinite faithful weights $\varphi_1$ and $\varphi_2$ on $e\cal{M}e$ and $e^\perp \cal{M} e^\perp$. We can define a normal semifinite faithful weight $\varphi$ on $\cal{M}$ by 
\begin{equation}
\label{Extension-weight1}
\varphi(x)
\ov{\mathrm{def}}{=} \varphi_1(exe)+\varphi_2(e^\perp x e^\perp), \quad x \in \cal{M}_+
\end{equation}
With \eqref{Charac-centralizer}, it is easy to check that $e$ belongs to the centralizer $\cal{M}^\varphi$ of $\varphi$ and we have $\varphi_e=\varphi_1$. 

%Let us construct a normal faithful positive linear form with centralizer containing $e$. 
%%\cite[10.10]{Str81} \cite[Exercise 7.6.46]{KaR97b} 
%Consider two normal faithful positive linear forms $\varphi_1$ and $\varphi_2$ on $e\cal{M}e$ and $e^\perp \cal{M} e^\perp$. By \cite[p.~155]{RaX03}, we can define a normal faithful positive linear form $\phi$ on $\cal{M}$ by 
%\begin{equation}
%\label{Extension-weight1}
%\phi(x)
%\ov{\mathrm{def}}{=} \varphi_1(exe)+\varphi_2(e^\perp x e^\perp), \quad x \in \cal{M}_+.
%\end{equation}
%Moreover, $e$ belongs to the centralizer of $\phi$ by \eqref{Charac-centralizer} and we have $\phi_e=\varphi_1$.

\paragraph{Some complements}
We will use the following classical lemma. We add the uniqueness part in part 1. 
\begin{lemma}
\label{Lemma-DP}
Let $\cal{M}$ be a (semifinite) von Neumann algebra equipped with a  normal semifinite faithful trace $\tau$. 
\begin{enumerate}
	\item Let $a,b \in \L^0(\cal{M},\tau)^+$ satisfy $a \leq b$. Then there exists a unique $x \in \mathcal{M}$ such that $a^{\frac 12} = x b^{\frac 12}$ and $xs(b)=x$. Moreover, we have necessarily $\norm{x}_\infty \leq 1$.

\item If $a \leq b$ in the space $\L^0(\cal{M},\tau)_\sa$ and if $x \in \L^0(\cal{M},\tau)$ then $x^*ax \leq x^*bx$.
\end{enumerate}
\end{lemma}

\begin{proof}
1. Suppose that the von Neumann algebra $\mathcal{M}$ acts on a Hilbert $H$. If $x$ satisfies the properties then $x$ is obviously null on the subspace $(1-s(b))H$. Moreover, the restriction of $x$ to the subspace $s(b)H$ is given by $x|_{s(b)H}=a^{\frac 12}b^{-\frac 12}$ where $b^{-\frac 12} \co s(b)H \to s(b)H$. We conclude that $x$ is uniquely determined. The proof of the existence is given in \cite[Remark 2.3]{DeJ04} and shows the relation $xs(b)=x$.
\end{proof}

The second lemma is an easy observation.

\begin{lemma}
\label{lemma2-GL}
Let $\cal{M}$ be a von Neumann algebra and  $1 \leq p <\infty$. Let $h$ be a positive element of the space $\L^p(\cal{M})$.
\begin{enumerate}
	\item The linear map $s(h)\cal{M}s(h) \to \L^p(\cal{M})$, $x \mapsto h^{\frac{1}{2}} x h^{\frac{1}{2}}$ is injective.
	\item Suppose that $1 \leq p <\infty$. The subspace $h^{\frac{1}{2}} \cal{M} h^{\frac{1}{2}}$ is dense in the space $s(h)\L^p(\cal{M})s(h)$ for the topology of $\L^p(\cal{M})$.
\end{enumerate}
\end{lemma}

\begin{proof}
1) If $h^{\frac{1}{2}} x h^{\frac{1}{2}}=0$ then $s(h)xs(h)=0$ by using \eqref{support-properties} twice. Since $x \in s(h)\cal{M}s(h)$, we conclude that $x=s(h)xs(h)=0$.

2) Suppose that $y \in \L^{p^*}(\cal{M})$ belongs to the annihilator $\big(h^{\frac{1}{2}} \cal{M} h^{\frac{1}{2}}\big)^\perp$ of the subspace $h^{\frac{1}{2}} \cal{M} h^{\frac{1}{2}}$ of $\L^p(\cal{M})$. For any $x \in \cal{M}$, we have
$\tr\big(h^{\frac{1}{2}} xh^{\frac{1}{2}}y\big)=0$, i.e.~$\tr\big(xh^{\frac{1}{2}}yh^{\frac{1}{2}}\big)=0$. Since $h^{\frac{1}{2}}yh^{\frac{1}{2}} \in \L^1(\cal{M})$, we deduce that $h^{\frac{1}{2}}yh^{\frac{1}{2}}=0$. We infer that $s(h)ys(h)=0$. It follows immediately that $y$ belongs to the annihilator $\big(s(h)\L^p(\cal{M})s(h)\big)^\perp$. The proof is complete.
\end{proof}

The following is a variant of \cite[Lemma 2.2 (d)]{Sch86}. Our approach is probably more transparent. See \cite{PeT72} and \cite[Proposition 3.1]{Kos81} for related facts.

\begin{lemma}
\label{lemma-DP-bis}
Let $\cal{M}$ be a von Neumann algebra. Suppose that $1 \leq p < \infty$. If $a,b \in \L^p(\cal{M})_+$ satisfy $a \leq b$ then there exists a unique $z \in s(b)\cal{M}s(b)$ such that $a=b^{\frac{1}{2}}zb^{\frac{1}{2}}$. Moreover, we have $0 \leq z \leq s(b)$.
\end{lemma}

\begin{proof}
Let $\L^p(\cal{M},\varphi)$ be a realization of $\L^p(\cal{M})$ in $\tilde{\cal{M}}= \cal{M} \rtimes_{\sigma^\varphi} \R$. By Lemma \ref{Lemma-DP}, there exists a unique $x \in \tilde{\cal{M}}$ such that $a^{\frac 12} = x b^{\frac 12}$ and $xs(b)=x$.  Note that $a^{\frac 12},b^{\frac 12} \in \L^{2p}(\cal{M})$. Then for any $s \in \R$
$$
\e^{-\frac s{2p}}a^{\frac 12}
\ov{\eqref{Def-Haagerup}}{=} \hat{\sigma}_s\big(a^{\frac 12}\big)
=\hat{\sigma}_s\big(x b^{\frac 12}\big)
=\hat{\sigma}_s(x)\hat\sigma_s\big( b^{\frac 12}\big)
\ov{\eqref{Def-Haagerup}}{=} \hat{\sigma}_s(x)\e^{-\frac s{2p}}b^{\frac 12}.
$$
Hence $a^{\frac 12}= \hat\sigma_s(x)b^{\frac 12}$. On the other hand, since $x=xs(b)$ and $s(b)\in \tilde{\cal{M}}$, we have
$$
\hat{\sigma}_s(x)
=\hat{\sigma}_s(xs(b))
=\hat{\sigma}_s(x)\hat{\sigma}_s(s(b))
\ov{\eqref{Def-Haagerup}}{=} \hat{\sigma}_s(x)s(b).
$$
Thus by uniqueness of $x$ we get $\hat{\sigma}_s(x)=x$ for every $s \in \R$. That means that $x \in \cal{M}$ by \eqref{carac-Pi-de-M}. By Lemma \ref{Lemma-DP}, we deduce that $\norm{x} \leq 1$.

From $a^{\frac 12} = x b^{\frac 12}$ and $\norm{x}_\infty \leq 1$ it follows readily that $a=b^{\frac 12}x^*xb^{\frac 12}$ and that $0 \leq x^*x \leq \norm{x^*x}_\infty=\norm{x}_\infty^2 \leq 1$ where we use \cite[1.6.9]{Dix77} in the second inequality. If we set $z \ov{\mathrm{def}}{=} s(b)x^*xs(b)$ then we have $z \in s(b)\cal{M}s(b)$. Moreover, since $s(b^{\frac 12})=s(b)$, we have $b^{\frac 12}zb^{\frac 12}=  b^{\frac 12}s(b)x^*xs(b)b^{\frac 12}=b^{\frac 12}x^*xb^{\frac 12}=a$ and $0 \leq z \leq s(b)$. 

The uniqueness of the element $x$ follows from Lemma \ref{lemma2-GL}. 
\end{proof}

%%%%%%%%%%%%%%%%%%%%%%%%%%%%%%%%%%%%%%%%%%%%%%%%%%%%%%%%%%%%%%%%%%%%%%%%%%%%%%%%%%%%%%%%%%%%%%%%%
\section{Conditional expectations on noncommutative $\L^p$-spaces}
\label{Sec-conditional}

\paragraph{Conditional expectations on von Neumann algebras} Recall that a positive map $T \co A \to A$ on a $\mathrm{C}^*$-algebra $A$ is said faithful if $T(x)=0$ for some $x \in A_+$ implies $x=0$. 

Let $B$ be a $\mathrm{C}^*$-subalgebra of a $\mathrm{C}^*$-algebra $A$. A map $\E \co A \to A$ is called a conditional expectation on $B$ if it is a positive projection \cite[p.~116]{Str81} of range $B$ which is $B$-bimodular, that is 
$$
\E(xyz)
=x\E(y)z, \quad  y \in A, x,z \in B,
$$
Such a map is completely positive \cite[Proposition p.~118]{Str81}.

Now, suppose that $A=\cal{M}$ is a von Neumann algebra, that $B=\cal{N}$ is a von Neumann subalgebra and that there exists a faithful normal conditional expectation $\E \co \cal{M} \to \cal{M}$ from $\cal{M}$ onto $\cal{N}$, and that $\psi$ is a normal semifinite faithful weight on $\cal{N}$. Then by \cite[Lemme 1.4.3]{Con73}, $\varphi \ov{\mathrm{def}}{=} \psi \circ \E$ is a normal semifinite faithful weight on $\cal{M}$ and the automorphisms groups of the two weights are linked by the relation
\begin{equation}
\label{eq:E-sigma}
\E \circ \sigma_t^\varphi
=\sigma_t^\psi \circ \E, \quad t \in \R.
\end{equation}
This implies that the von Neumann algebra $\cal{N}$ is invariant under $\sigma^\varphi$, i.e.~$\sigma_t^\varphi(\cal{N}) \subset \cal{N}$ for all $t \in \R$, and that $\sigma_t^\psi \co \cal{N} \to \cal{N}$ is the restriction of $\sigma_t^\varphi \co \cal{M} \to \cal{M}$ to $\cal{N}$. 

\paragraph{$\L^p$-extensions of positive maps} Let $\cal{M}$ be a von Neumann algebra equipped with a normal faithful positive linear form $\varphi$ with associated density operator $D_{\varphi}$. If $1 \leq p <\infty$, note that by Lemma \ref{lemma2-GL} (see also \cite[Lemma 1.1]{JuX03} and \cite[Corollary 4]{Wat88}), $D_\varphi^{\frac{1}{2p}} \cal{M} D_\varphi^{\frac{1}{2p}}$ is a dense subspace of the space $\L^p(\cal{M})$.  Consider a unital positive map $T \co \cal{M} \to \cal{M}$ and assume that $\psi(T(x)) =\varphi(x)$ for any $x \in \cal{M}_+$. Define
\begin{equation}
\label{Map-extension-Lp}
\begin{array}{cccc}
  T_p  \co &  D_\varphi^{\frac{1}{2p}} \cal{M} D_\varphi^{\frac{1}{2p}}  &  \longrightarrow   &  D_\varphi^{\frac{1}{2p}} \cal{M} D^{\frac{1}{2p}}_{\varphi}  \\
           &   D_\varphi^{\frac{1}{2p}}x D_\varphi^{\frac{1}{2p}}  & \longmapsto &  D_\varphi^{\frac{1}{2p}}T(x) D_\varphi^{\frac{1}{2p}}  \\
\end{array}.
\end{equation}
By \cite[Theorem 5.1]{HJX10}, the map $T_p$ extends to a contractive map from $\L^p(\cal{M})$ into $\L^p(\cal{M})$. We can use this result for obtaining $\L^p$-extension of conditional expectations.

\paragraph{$\L^p$-extensions of conditional expectations in the general case}
Now, we will describe an approach for the general case. We use the notations in the first paragraph of this section. Note the crossed product $\tilde{\cal{N}} \ov{\mathrm{def}}{=} \cal{N} \rtimes_{\sigma^\varphi} \R$ is a $*$-subalgebra of $\tilde{\cal{M}} \ov{\mathrm{def}}{=} \cal{M} \rtimes_{\sigma^\varphi} \R$, and the natural representation of $\cal{N}$ in $\tilde{\cal{N}}$ is the restriction of that of $\cal{M}$ in $\tilde{\cal{M}}$. It is clear that the dual action $\hat\sigma^\varphi$ of $\R$ on $\cal{M}$ restricts to $\hat\sigma^\phi$ on $\cal{N}$. 

Let now $\hat \E$ be the restriction of $\Id_{\B(\L^2(\R))} \ot \E$ to $\tilde{\cal{M}}$, when $\tilde{\cal{M}}$ is considered as a von Neumann subalgebra of $\B(\L^2(\R)) \otvn \cal{M}$. Note that $\hat\E$ is normal and faithful since $\Id \ot \E$ is normal and faithful. Since $\hat \E$ preserves the operators $\lambda_s$, $s \in \R$ and sends $\pi(\cal{M})$ onto $\pi(\cal{N})$ (due to relation \eqref{eq:E-sigma}) it sends $\tilde{\cal{M}}$ onto $\tilde{\cal{N}}$. The $\tilde{\cal{M}}$-bimodularity of $\hat{\E}$ results from the fact that $\Id_{\B(\L^2(\R))} \ot \E$ is $\B(\L^2(\R)) \otvn \cal{N}$-bimodular. We obtain therefore a normal faithful conditional expectation $\hat{\E}$ from $\tilde{\cal{M}}$ onto $\tilde{\cal{N}}$. 
 
\begin{lemma}
\label{rel:dual-weights}
The dual weights $\hat{\varphi}$, $\hat{\phi}$ on $\tilde{\cal{M}}$, resp. $\tilde{\cal{N}}$ to the weights $\varphi$, $\phi$ satisfy the relation
\begin{equation}
\label{lien-dual-weights}
\hat{\varphi} 
=\hat\phi \circ \hat{\E}.
\end{equation}
\end{lemma}

\begin{proof}
It is easy to see that each automorphism $\hat \sigma_s$ commutes with $\hat \E$. It follows from \eqref{Operator-valued} that the operator valued weights $T_{\tilde{\cal{M}}}$ and $T_{\tilde{\cal{N}}} $ verify the relation
\begin{equation}
\label{entrelace-operator-weight}
T_{\tilde{\cal{N}}} \circ \hat \E 
=\bar \E \circ  T_{\tilde{\cal{M}}}
\end{equation}
where $\bar \E \co \bar{\cal{M}}^+ \to \bar{\cal{N}}^+$ is the natural extension of the operator $\E \co \cal{M}^+ \to \cal{N}^+$. Let us also denote by $\bar\omega$ the natural extension to $\bar{\cal{M}}_+$ of any normal weight $\omega$ on $\cal{M}$ \cite[Proposition 1.10]{Haa79a}. Then
$$
\hat{\varphi}
\ov{\eqref{Def-poids-dual}}{=} \bar \varphi\circ T_{\tilde{\cal{M}}}
=\ovl{\phi\circ \E}\circ T_{\tilde{\cal{M}}}
=\bar\phi \circ \bar\E \circ T_{\tilde{\cal{M}}}
\ov{\eqref{entrelace-operator-weight}}{=} \bar\phi \circ T_{\tilde{\cal{N}}} \circ \hat \E
\ov{\eqref{Def-poids-dual}}{=} \hat\phi \circ \hat \E.
$$
\end{proof}

Let $\tau_{\tilde{\cal{N}}}$ be the canonical trace on the crossed product $\tilde{\cal{N}}$. Using \cite[Lemme 1.4.4]{Con73} in the second equality, we obtain
$$
\big (\D \hat{\varphi} : \D(\tau_{\tilde{\cal{N}}} \circ \hat \E)\big)_t
\ov{\eqref{lien-dual-weights}}{=} \big (\D(\hat\phi\circ \hat \E) : \D(\tau_{\tilde{\cal{N}}} \circ \hat \E)\big)_t 
=\big(\D\hat{\phi}:\D\tau_{\tilde{\cal{N}}}\big)_t
=\lambda_t=\big(\D\hat{\varphi}:\D\tau_{\tilde{\cal{M}}}\big)_t
, \quad t \in \R. 
$$
By \cite[p.~49]{Str81}, we deduce that $\tau_{{\tilde{\cal{N}}}} \circ \hat \E$ is the canonical trace $\tau_{\tilde{\cal{M}}}$ on the crossed product $\tilde{\cal{M}}$. In particular, $\tau_{\tilde{\cal{N}}}$ is the restriction of $\tau_{\tilde{\cal{M}}}$. From now on, we set $\tau=\tau_{\tilde{\cal{M}}}$, and $\tau_{\tilde{\cal{N}}}=\tau|_{\tilde{\cal{N}}}$.

It follows that the inclusion $\tilde{\cal{N}} \subset \tilde{\cal{M}}$ extends to a topological linear embedding $\L^0(\tilde{\cal{N}},\tau_{\tilde{\cal{N}}}) \subset \L^0(\tilde{\cal{M}},\tau_{\tilde{\cal{M}}})$ (which respects the operations of $*$-algebras). Using the characterization \eqref{Def-Haagerup} of Haagerup $\L^p$-spaces, it is easy to see that for $1 \leq p \leq \infty$,
\begin{equation}
\label{eq:embed-Lp}
\L^p(\cal{N},\phi) 
=\L^p(\cal{M},\varphi) \cap \L^0(\tilde{\cal{N}},\tau).
\end{equation}
Note that the external products and Mazur maps in the scale $(\L^p(\cal{N},\varphi))$ are inherited from the scale $(\L^p(\cal{M},\varphi))$ and that the inclusion $i_p$ of $\L^p(\cal{N},\phi)$ into the space $\L^p(\cal{M},\varphi)$ is necessarily isometric, in virtue of the formula 
$$
\norm{h}_p
= \tau\big(\chi_{[1,+\infty)}(|h|)^p\big)^{\frac{1}{p}}
$$
(see \cite[Ch.~II, Lemma 5]{Terp81} for the case $p=1$). The map $i_p$ is also positive since 
\begin{align*}
\MoveEqLeft
\L^p(\cal{N},\phi)_+ 
= \L^p(\cal{N},\phi) \cap \L^0(\tilde{\cal{N}},\tau)_+ \\
&= \L^p(\cal{M},\varphi) \cap \L^0(\tilde{\cal{N}},\tau)_+ \subset \L^p(\cal{M},\varphi)\cap \L^0(\tilde{\cal{M}},\tau)_+ 
= \L^p(\cal{M},\varphi)_+.
\end{align*}
In fact $i_p$ is isometric and completely positive, since for each integer $n \in \nn$, the linear map $\mathrm{id}_{S_n^p}\ot i_p \co S^p_n(\L^p(\cal{N},\phi)) \to S^p_n(\L^p(\cal{M},\varphi))$ identifies to the inclusion map $\L^p(\M_n(\cal{N}), \phi_n)\subset \L^p(\M_n(\cal{M}),\varphi_n)$
with $\phi_n \ov{\mathrm{def}}{=} \tr_n \ot \phi$, $\varphi_n \ov{\mathrm{def}}{=} \tr_n\ot\varphi$ where $\tr_n$ is the ordinary trace on the algebra $\M_n$ of $n \times n$ matrices.

On the other hand there is a natural isometric linear injection of $\cal{N}_*$ into $\cal{M}_*$, namely the map $\E_* \co \omega \mapsto \omega \circ\E$. It turns out that in the identification of preduals with $\L^1$-spaces, this maps becomes exactly the inclusion map of $\L^1(\cal{N},\phi)$ into $\L^1(\cal{M},\varphi)$, that is
\begin{align}
\label{inclus-L1}
h_\omega^{\phi}
=h_{\omega\circ\E}^\varphi
\end{align}
for every $\omega \in \cal{N}_*$. Indeed, using \cite[Lemme 1.4.4]{Con73} in the third  equality, we have
$$
\big(h_{\omega\circ\E}^\varphi\big)^{\i t} 
\ov{\eqref{Radon-Nikodym-1}}{=} \big(\D\widehat{\omega\circ\E} : \D\tau\big)_t 
=\big(\D(\hat\omega\circ \hat{\E}):\D(\tau|_\N\circ \hat{\E})\big)_t 
=\big(\D\hat\omega: \D\tau|_\N\big)_t 
\ov{\eqref{Radon-Nikodym-1}}{=} (h_\omega^{\phi})^{\i t}.
$$
It follows easily that the functionals $\tr_\varphi$ on $\L^1(\cal{M},\varphi)$ and $\tr_\phi$ on $\L^1(\cal{N},\phi)$ coincide on $\L^1(\cal{N},\phi)$:
\begin{equation}
\label{Traces-coincident}
\tr_\varphi|_{\L^1(\cal{N},\phi)}
=\tr_\phi.
\end{equation}
Similarly the restriction map $R \co \cal{M}_* \to \cal{N}_*, \omega\mapsto \omega\big|_\cal{N}$, identifies to the map $\E_1 \co \L^1(\cal{M},\varphi)\to \L^1(\cal{N},\phi)$ defined by
\[
\E_1(h_\omega^\varphi)
=h_{\omega\mid _N}^\phi
= h_{\omega \circ \E}^\varphi
\]
for every $\omega \in \M_*$, which means that 
\begin{align}\label{E1}
\tr\big(\E_1(h)x\big)= \tr\big(h\E(x)\big), \quad h \in \L^1(\cal{M},\varphi), x \in \cal{M}
 \end{align}
that is, $\E$ is conjugate to $\E_1$ in the duality of the trace functional $\tr$. Note that in particular $\tr=\tr \circ\,\E_1$ (take $x=I$). It is also easy to see that the map $\E_1$ is $\cal{N}$-bimodular, that is 
$$
\E_1(xhy)
=x\E_1(h)y, \quad  x,y \in \cal{N}, h \in \L^1(\cal{M},\varphi).
$$

Now for $1<p<\infty$ and $h \in \L^p(\cal{M},\varphi)$ we define $\E_p(h)$ as the unique element of $\L^p(\cal{N},\phi)$ such that 
\begin{align}
\label{Ep}
\tr\big(\E_p(h)k\big) 
= \tr(hk), \quad k \in \L^{p^*}(\cal{N},\phi)
 \end{align}
where $\frac{1}{p}+\frac{1}{p^*}=1$. That $\E_p(h)$ exists and is unique stems directly from the fact that $\L^p(\cal{N})$ is the conjugate space to $\L^{p^*}(\cal{N})$. It is then easy to see that $\E_p$ is a contractive linear projection in $\L^p(\cal{M},\varphi)$, which is positive (since the positive cone in $\L^{p^*}(\cal{M})$ is polar to that in $\L^p(\cal{M})$, \cite[Ch.~II, Proposition 33]{Terp81}). For any $h \in \L^p(\cal{M},\varphi)$ and any $k \in \L^{p^*}(\cal{M},\varphi)$ we have
$$
\tr\big(\E_p(h)k\big)
=\tr\big(\E_p(h)\E_{p^*}(k)\big)
=\tr\big(h\,\E_{p^*}(k)\big)
$$
thus the conjugate of $\E_p$, viewed as a map from $\L^p(\cal{M},\varphi)$ into itself, is $\E_{p^*}$. Moreover, if $\frac{1}{p}+\frac{1}{q}=\frac{1}{r} \leq 1$, $h \in \L^p(\cal{M},\varphi)$ and $k \in \L^q(\cal{N},\phi)$, it is easily seen that 
\begin{align}
\label{esp-cond-prop}
\E_r(hk)=\E_p(h)k 
\quad \text{and} \quad 
\E_r(kh)=k\E_p(h).
\end{align}
In particular with $q=\infty$ we get that $\E_p$ is $\cal{N}$-bimodular. 

Let us point a more abstract way to characterize $\E_p$. For any $1 \leq p \leq \infty$, let $i_p \co \L^p(\cal{N}) \to \L^p(\cal{M})$ be the embedding associated with the weight $\varphi$ and the $\varphi$-invariant conditional expectation $\E$. Then for $1 \leq p \leq \infty$ we have
\begin{align}
\label{second-def-Ep}
\E_p 
= i_p\circ (i_{p^*})^* 
\  \text{ if } 1\leq p\leq \infty, \ 
\E_1 = i_1\circ (i_\infty)_*
\end{align}
where $(i_\infty)_*$ is the preadjoint of $i_\infty$. %In particular for $1\le p<\infty$ we have $(\E_p)^* = \E_{p^*}$ (with the convention that $\E_\infty= \E$).

\begin{remark}
\label{Proba-case} \normalfont
In the case where $\varphi$ is a normal linear functional, the previous defined conditional expectations $\E_p$ satisfies the formula:
\begin{equation}\label{Def-esperance-h-bis}
\E_p\big(D_{\varphi}^{\frac{1}{2p}} x D_{\varphi}^{\frac{1}{2p}}\big)
%\ov{\eqref{Map-extension-Lp}}{=}
= D_{\varphi}^{\frac{1}{2p}}\E(x)D_{\varphi}^{\frac{1}{2p}},\quad x \in \cal{M}. 
\end{equation}
 In particular if $\varphi$ is a normal state, they coincide with the conditional expectations $\mathcal E_p$ defined in \cite[Section 2]{JuX03}.
\end{remark}

\begin{proof}
In this case, the density operators $D_\varphi$  and $D_\phi$ associated with the weights $\varphi$ and $\phi$ belong respectively to the spaces $\L^1(\cal{M},\varphi)$ and $\L^1(\cal{N},\phi)$. Moreover by \eqref{inclus-L1} they obviously coincide modulo the embedding of the space $\L^1(\cal{N},\phi)$ into $\L^1(\cal{M},\varphi)$ previously described in \eqref{eq:embed-Lp}. Thus $D_\varphi$ belongs to the space $\L^1(\cal{N},\phi)$. Hence $D_\varphi^{\frac{1}{2p}} \in \L^{2p}(\cal{N},\phi)$ and the equation \eqref{Def-esperance-h-bis} is then an immediate consequence of the bimodularity property \eqref{esp-cond-prop}. The map $\cal{E}_p$ in \cite[Section 2]{JuX03} coincides thus with $\E_p$ on the subspace  $D_{\varphi}^{\frac{1}{2p}}\cal{M}D_{\varphi}^{\frac{1}{2p}}$, which is dense in the Banach space $\L^p(\cal{M},\varphi)$ by Lemma \ref{lemma2-GL}. These maps are both linear and contractive (for $\mathcal E_p$, see \cite[Lemma 2.2]{JuX03}), they coincide thus everywhere.
\end{proof}

It is well known, at least in the <<probabilistic case>>,  that the linear maps $\E_p$ are completely positive and contractive, see e.g.~\cite[p.~92]{JRX05}. The proof there is by the usual tensor product argument, and it works as well with our definition of $\E_p$ in the general case. But in our context these properties may be derived directly from the relations \eqref{second-def-Ep} and the complete positivity of the isometric maps $i_p$. Finally, note that an expanded unpublished version of \cite{HJX10} contains a more general approach of extension of positive maps than the one of \eqref{Map-extension-Lp}.

%%%%%%%%%%%%%%%%%%%%%%%%%%%%%%%%%%%%%%%%%%%%%%%%%%%%%%%%%%%%%%%%%%%%%%%%%%%%%%%%%%%%%%%
\section{Projections on von Neumann algebras}
\label{Sec-projections-von}

Let $A$ be a $\mathrm{C}^*$-algebra. Recall that a linear map $T \co A \to A$ is a Schwarz map \cite[Definition 9.9.5 p.~1023]{Pal01} if we have $T(x)^*T(x) \leq T(x^*x)$ for any $x \in A$. Note that a Schwarz map is positive. By \cite[Corollary 1.3.2 p.~9]{Sto13}, any 2-positive contraction $T \co A \to A$ on a unital $\mathrm{C}^*$-algebra $A$ is a Schwarz map. Moreover, by \cite[Theorem 2.2.2 (2) p.~17]{Sto13}, if $P \co A \to A$ is a faithful projection map which is a Schwarz map then the range $\Ran P$ of $P$ is a $\mathrm{C}^*$-subalgebra of $A$. Finally, recall that a faithful projection on a unital $\mathrm{C}^*$-algebra is necessarily unital by \cite[Lemma 2.2.3 p.~18]{Sto13}.

\begin{prop}
\label{Proj-cp-VNA}
Let $\cal{M}$ be a von Neumann algebra. If $P \co \cal{M} \to \cal{M}$ is a normal faithful projection which is a Schwarz map, then $P$ is a conditional expectation and $\Ran P$ is a von Neumann subalgebra.
\end{prop}

\begin{proof}
By the previous result, the range $\Ran P$ is a unital $\mathrm{C}^*$-subalgebra of $\cal{M}$. Since $P$ is normal, $\Ran P=\ker(\Id_\cal{M}-P)$ is weak* closed and thus $\Ran P$ is a von Neumann subalgebra of $\cal{M}$. Note that $\norm{P}=\norm{P(1)}=\norm{1}=1$ since $P$ is positive. Using Tomiyama's theorem \cite[Theorem p.~116]{Str81}, we conclude that $P$ is a (normal faithful) conditional expectation.
\end{proof}

%%%%%%%%%%%%%%%%%%%%%%%%%%%%%%%%%%%%%%%%%%%%%%%%%%%%%%%%%%%%%%%%%%%%%%%%%%%%%%%%%%%%%%%%%%%%%%%%%%%%%
\section{Positive linear maps between noncommutative $\L^p$-spaces}
\label{Positive-maps}

Our main tool will be the following extension of \cite[Theorem 3.1]{JRX05}. The previous lemmas of Section \ref{Haagerup-noncommutative} allows us to remove some assumptions in \cite[Theorem 3.1]{JRX05}. Since the proof of \cite[Theorem 3.1]{JRX05} unfortunately contains some gaps and misleading points\footnote{\thefootnote. The complete positivity is not proved. The argument for the contractivity is ineffective.} and since this result is fundamental for the sequel, we give full details.  

\begin{thm}
\label{Th-relevement-cp} 
Let $\cal{M}$ and $\cal{N}$ be von Neumann algebras. Suppose that $1 \leq p < \infty$. Let $T \co \L^p(\cal{M}) \to \L^p(\cal{N})$ be a positive linear map. Let $h$ be a positive element of $\L^p(\cal{M})$. 
Then there exists a unique linear map $v \co \cal{M} \to s(T(h))\cal{N}s(T(h))$ such that 
\begin{equation}
\label{equa-relevement}
T\big(h^{\frac{1}{2}}xh^{\frac{1}{2}}\big)
=T(h)^{\frac{1}{2}}v(x)T(h)^{\frac{1}{2}},\qquad x \in \cal{M}.
\end{equation}
Moreover, this map $v$ is unital, contractive, and normal. If $T$ is $n$-positive for some $1 \leq n \leq \infty$, then $v$ is also $n$-positive.
\end{thm}

\begin{proof}
We let $e \ov{\mathrm{def}}{=} s(T(h))$ be the support projection of the element $T(h)$ of $\L^p(\cal{N})$. For any $x \in \cal{M}_+$, %by \cite[1.6.9]{Dix77}, 
since $0 \leq x \leq \norm{x}_\infty$, we have by the second part of Lemma \ref{Lemma-DP} that $0 \leq h^\frac{1}{2} x h^\frac{1}{2} \leq \norm{x}_\infty h$. Using the positivity of $T$, we obtain that
$$
0 
\leq T\big(h^\frac{1}{2} x h^\frac{1}{2}\big) 
\leq \norm{x}_\infty T(h).
$$
If $x \not=0$, note that $s(\norm{x}_\infty T(h))=s(T(h))=e$. Hence, for any $x \in \cal{M}_+$, there exists by Lemma \ref{lemma-DP-bis} a unique element $z(x) \in e\cal{N}e$ satisfying
$$
T\big(h^{\frac{1}{2}}xh^{\frac{1}{2}}\big)
=\big(\norm{x}_\infty T(h)\big)^{\frac{1}{2}} z(x)\big(\norm{x}_\infty T(h)\big)^{\frac{1}{2}}.
$$
Moreover, we have $0 \leq z(x) \leq e$. Note that $z(1)=e$ since $e=s((T(h))^{\frac{1}{2}})$. For any $x \in \cal{M}_+$, we let $v(x) \ov{\mathrm{def}}{=} \norm{x}_\infty z(x)$. So $\norm{v(x)}_\infty \leq \norm{x}_\infty \norm{z(x)}_\infty \leq \norm{x}_\infty \norm{e}_\infty \leq \norm{x}_\infty$. Furthermore, we have
\begin{equation}
\label{Fund-relation}
T\big(h^{\frac{1}{2}}xh^{\frac{1}{2}}\big)
=T(h)^{\frac{1}{2}}v(x)T(h)^{\frac{1}{2}}
\end{equation}
and $v(1)=e$. Let us show that $v \co \cal{M}_+ \to (e\cal{N}e)_+$ is additive. For any $x,y \in \cal{M}_+$, we have
\begin{align*}
T(h)^{\frac{1}{2}}v(x+y)T(h)^{\frac{1}{2}} 
&\ov{\eqref{Fund-relation}}{=} T\big(h^{\frac{1}{2}}(x+y)h^{\frac{1}{2}}\big)
=T\big(h^{\frac{1}{2}}xh^{\frac{1}{2}}\big)+T\big(h^{\frac{1}{2}}yh^{\frac{1}{2}}\big) \\
&\ov{\eqref{Fund-relation}}{=}T(h)^{\frac{1}{2}}v(x)T(h)^{\frac{1}{2}}+T(h)^{\frac{1}{2}}v(y)T(h)^{\frac{1}{2}}
=T(h)^{\frac{1}{2}}\big(v(x)+v(y)\big)T(h)^{\frac{1}{2}}.
\end{align*}
We conclude with Lemma \ref{lemma2-GL} that $v(x+y)=v(x)+v(y)$. By a standard reasoning, see e.g.~\cite[Lemma 1.26 p.~24]{AlT07}, the map $v$ extends uniquely to a real linear positive map $v \co \cal{M}_\sa \to (e\cal{N}e)_\sa$. We may extend it to a positive complex linear map from $\cal{M}$ into $e\cal{N}e$ by letting $v(x+\i y) =v(x) + \i v(y)$. As a positive and unital map ($v(1)=e$) between $\mathrm{C}^*$-algebras, $v$ is contractive by \cite[Corollary 2.9 p.~15]{Pau02}. The equation \eqref{equa-relevement} follows by linearity from the case $x \geq 0$. 
 
The uniqueness of $v$ is a consequence of the first part of Lemma \ref{lemma2-GL} applied with $T(h)$ and $\cal{N}$ instead of $h$ and $\cal{M}$.

Now, we prove that the map $v$ is normal. Since $1 \leq p< \infty$, we may consider the adjoint $T^* \co \L^{p^*}(\cal{N}) \to \L^{p^*}(\cal{M})$. We define the element $k\ov{\mathrm{def}}{=} T(h)^p$ of $\L^1(\cal{M})_+$ and define a linear map $w \co k^\frac{1}{2} \cal{N}k^\frac{1}{2} \to \L^1(\cal{M})$ on the dense subspace $k^\frac{1}{2} \cal{N} k^\frac{1}{2}$ of $\L^1(e\cal{N}e)$ (see Lemma \ref{lemma2-GL}) by
\begin{equation}
\label{Def-de-w}
w(k^\frac{1}{2} yk^\frac{1}{2}) 
\ov{\mathrm{def}}{=} h^\frac{1}{2} T^*\big(k^\frac{1}{2p^*} yk^\frac{1}{2p^*}\big)h^\frac{1}{2}, \quad y\in \cal{N}.
\end{equation}
If $x \in \cal{M}$, we have
\begin{align}\label{transpose}
\notag &\tr\big(w(k^\frac{1}{2} yk^\frac{1}{2}) x\big)           
\ov{\eqref{Def-de-w}}{=}\tr\big(h^\frac{1}{2} T^*\big(k^\frac{1}{2p^*} yk^\frac{1}{2p^*}\big)h^\frac{1}{2} x\big)
=\tr\big( T^*\big(k^\frac{1}{2p^*} yk^\frac{1}{2p^*}\big)h^\frac{1}{2} xh^\frac{1}{2}\big)\\
&=\tr\big(k^\frac{1}{2p^*} yk^\frac{1}{2p^*}T(h^\frac{1}{2} xh^\frac{1}{2})\big)
\ov{\eqref{equa-relevement}}{=} \tr\big(k^\frac{1}{2p^*} yk^\frac{1}{2p^*}T(h)^{\frac{1}{2}}v(x)T(h)^{\frac{1}{2}}\big)
=\tr\big(k^\frac{1}{2} yk^\frac{1}{2}v(x)\big).
\end{align} 
It follows that
\begin{align*}
	\left|\tr \big(w(k^\frac{1}{2} yk^\frac{1}{2})x\big) \right|
	&= \left|\tr\big(k^\frac{1}{2}y k^\frac{1}{2} v(x)\big) \right|
\leq \bnorm{k^\frac{1}{2} y k^\frac{1}{2}}_{1} \norm{v(x)}_\infty
\leq \bnorm{k^\frac{1}{2} y k^\frac{1}{2}}_{1} \norm{x}_\infty.
\end{align*} 
Therefore $w$ extends to a contraction from the space $\L^1(e\cal{N}e)$ into the space $\L^1(\cal{M})$. 
Furthermore, by \eqref{transpose} we conclude that $w^* = v$. Hence $v$ is normal.

Assume that $T$ is $n$-positive. Using the first part of the proof with the positive map $T^{(n)} \ov{\mathrm{def}}{=} \Id_{S^p_n} \ot T \co S^p_n(\L^p(\cal{M}))\to S^p_n(\L^p(\cal{N}))$ and by replacing $h$ with $\I_n \ot h$ whose projection support is $s(\I_n \ot h)=s(\I_n) \ot s(h) = \I_n \ot e$, we see that there exists a unique linear map $v_n \co \M_n(\cal{M}) \to (\I_n \ot e)\M_n(\cal{N})(\I_n \ot e)=\M_n(e\cal{N}e)$ such that for any $[x_{ij}] \in \M_n(\cal{M})$ we have
\begin{equation*}
T^{(n)}\left((\I_n \ot h)^{\frac{1}{2}}
[x_{ij}](\I_n \ot h)^{\frac{1}{2}}\right)\\
=\big(T^{(n)}(\I_n \ot h)\big)^{\frac{1}{2}}
v_n\left([x_{ij}]\right)
\big(T^{(n)}(\I_n \ot h)\big)^{\frac{1}{2}}.
\end{equation*}
Moreover, this map is unital, contractive, and normal. Note that
\begin{align*}
\MoveEqLeft
T^{(n)}\left((\I_n \ot h)^{\frac{1}{2}}
[x_{ij}](\I_n \ot h)^{\frac{1}{2}}\right) 
=T^{(n)}\left(\begin{bmatrix}
 h    &  &   \\
     & \ddots & \\
     &  &  h  \\
\end{bmatrix}^{\frac{1}{2}}
\begin{bmatrix}
 x_{11}    & \cdots & x_{1n}  \\
  \vdots   &  & \vdots	\\
  x_{n1}   & \cdots &  x_{nn}  \\
\end{bmatrix}\begin{bmatrix}
 h    &  &   \\
     & \ddots & \\
     &  &  h  \\
\end{bmatrix}^{\frac{1}{2}}\right)  \\         
&=T^{(n)}\left(\begin{bmatrix}
  h^{\frac{1}{2}} x_{11} h^{\frac{1}{2}}   & \cdots & h^{\frac{1}{2}} x_{1n}h^{\frac{1}{2}}  \\
\vdots	&&\vdots\\
 h^{\frac{1}{2}}  x_{n1} h^{\frac{1}{2}}  & \cdots &h^{\frac{1}{2}} x_{nn} h^{\frac{1}{2}}   \\
\end{bmatrix}\right) 
=\begin{bmatrix}
  T(h^{\frac{1}{2}} x_{11} h^{\frac{1}{2}})   & \cdots &T(h^{\frac{1}{2}}  x_{1n}h^{\frac{1}{2}})  \\
\vdots	&& \vdots\\
 T(h^{\frac{1}{2}}  x_{n1} h^{\frac{1}{2}})  & \cdots &T(h^{\frac{1}{2}}  x_{nn} h^{\frac{1}{2}})   \\
\end{bmatrix} \\
&\ov{\eqref{equa-relevement}}{=} \begin{bmatrix}
  T(h)^{\frac{1}{2}}v(x_{11})T(h)^{\frac{1}{2}}   & \cdots&T(h)^{\frac{1}{2}}v(x_{1n})T(h)^{\frac{1}{2}}  \\
\vdots &&\vdots	\\
  T(h)^{\frac{1}{2}}v(x_{n1})T(h)^{\frac{1}{2}}  &  \cdots &T(h)^{\frac{1}{2}}v(x_{2nn})T(h)^{\frac{1}{2}} \\
\end{bmatrix}
\end{align*} 
\begin{align*} 
\MoveEqLeft
=\begin{bmatrix}
 T(h)    &  &   \\
     & \ddots & \\
     &  &  T(h)  \\
\end{bmatrix}^{\frac{1}{2}}
\begin{bmatrix}
 v(x_{11})    & \cdots & v(x_{1n})  \\
  \vdots   &  & \vdots	\\
  v(x_{n1})   & \cdots &  v(x_{nn})  \\
\end{bmatrix}
\begin{bmatrix}
 T(h)    &  &   \\
     & \ddots & \\
     &  &  T(h)  \\
\end{bmatrix}^{\frac{1}{2}} \\
&=\big(T^{(n)}(\I_n \ot h)\big)^{\frac{1}{2}}
(\Id_{\M_n} \ot v)\left([x_{ij}]\right)
\big(T^{(n)}(\I_n \ot h)\big)^{\frac{1}{2}}.
\end{align*} 
Consequently, by uniqueness, we conclude that $\Id_{\M_n} \ot v=v_n$. Hence $v$ is $n$-positive.
\end{proof}

\section{Support projections of contractive projections}
\label{From-local-to-global}
%%%%%%%%%%%%%%%%%%%%%%%%%%%%%%%%%%%%%%%%%%%%%%%%%%%%%%%%%%%%%%%%%%%%%%%%%%%%%%%%%%%%%%

Suppose that $1 \leq p<\infty$. Let $\cal{M}$ be a $\sigma$-finite ($=$ countably decomposable) von Neumann algebra and $P \co \L^p(\cal{M}) \to \L^p(\cal{M})$ be a positive contractive projection. We define the support projection $s(P)$ of $\Ran P$ as the supremum in $\cal{M}$ of the supports of the positive elements of $\Ran P$:
\begin{equation}
\label{support-s(P)}
s(P)
\ov{\mathrm{def}}{=} \bigvee_{h \in \Ran P, h \geq 0} s(h).
\end{equation}

\begin{prop}
\label{Prop-sh=sP}
Suppose that $1 \leq p<\infty$. Let $\cal{M}$ be a $\sigma$-finite von Neumann algebra and $P \co \L^p(\cal{M}) \to \L^p(\cal{M})$ be a positive contractive projection. Then there exists a positive element $h$ of $\Ran P$ such that $s(h)=s(P)$.
\end{prop}

\begin{proof}
We note first that for every at most countable family $(h_i)_{i\in I}$ of positive elements in $\Ran P$ there is a positive element $h$ in $\Ran P$ such that $s(h) \geq s(h_i)$ for all $i \in I$. Indeed assuming $\norm{h_i}_p \leq 1$ for all $i$, take simply $h=\sum\limits_{i \in I} 2^{-i}h_i$. By \cite[Exercise 7.6.46 p.~500]{KaR97}, we can consider a normal faithful state $\varphi$ on $\cal{M}$. We introduce the positive real number 
$$
a
\ov{\mathrm{def}}{=} \sup\big\{\varphi(s(h)): h \in \Ran P, h \geq 0\big\}.
$$ 
This supremum is attained. Indeed, consider a positive sequence $(h_n)_{}$ of positive elements in $\Ran P$ such that $\varphi(s(h_n)) \uparrow a$, then any positive element $h \in \Ran P$ such that $s(h) \geq s(h_n)$ for all $n$ satisfies $\varphi(s(h))= a$. 

Fix such an element $h \in \Ran P$ with $h \geq 0$ such that $\varphi(s(h))=a$. We have $s(h) \leq s(P)$. If $s(h) \not= s(P)$, there exists $h' \in \Ran P$ with $h' \geq 0$ such that $s(h') \not\leq s(h)$. This implies that 
$$
s(h+h')
\geq s(h)
\vee s(h')>s(h).
$$ 
Hence $\varphi(s(h+h')) > \varphi(s(h))$ which is impossible by maximality of $\varphi(s(h))$. Thus $s(h)=s(P)$ as desired.
\end{proof}

If $k \in \Ran P$, we can write $k=k_1-k_2+\i(k_3-k_4)$ with $k_1,k_2,k_3,k_4 \geq 0$. Hence $k=P(k_1)-P(k_2)+\i(P(k_3)-P(k_4))$ where $P(k_1),P(k_2),P(k_3),P(k_4) \geq 0$ since  $P$ is positive. Clearly, this implies that $s(P)ks(P)=k$ and consequently $\Ran P=s(P) \Ran P\, s(P)$. Moreover, we have 
$$
\Ran P= P(\Ran P)
=P\big(s(P)\,\Ran P\, s(P)\big)
\subset  P\big(s(P)\L^p(\cal{M}) s(P)\big)\subset \Ran P.
$$
Finally, we obtain 
\begin{equation}
\label{Inter-200}
\Ran P
=P\big(s(P)\L^p(\cal{M})s(P)\big)
=P\big(s(h)\L^p(\cal{M})s(h)\big). 
\end{equation}  

The third part of Theorem \ref{main-th-ds-intro} is a consequence of a more general fact on contractive projections acting on noncommutative $\L^p$-spaces associated to \textit{arbitrary} von Neumann algebras. To state the result, we introduce the left and right support projections of a contractive projection $P \co \L^p(\cal{M}) \to \L^p(\cal{M})$, that we denote respectively by $s_\ell(P)$ and $s_r(P)$. Indeed, we let
\begin{equation}
\label{left_right_supp}
s_\ell(P)
\ov{\mathrm{def}}{=} \bigvee_{h \in \Ran P} s_\ell (h),\quad 
s_r(P) 
\ov{\mathrm{def}}{=} \bigvee_{h \in \Ran P} s_r (h).
\end{equation}
Clearly if $P$ is positive then $P$ is an adjoint preserving map and consequently $s_\ell(P) = s_r(P)$ and these supports coincide with the support projection $s(P)$ defined in \eqref{support-s(P)}. Our goal is to prove Proposition \ref{prop:range-pc} which describes the link between $P$ and its left and right support projections. We need some information about normalized duality mappings.

We will use the following result \cite[Theorem 6]{CoS70} of Cohen and Sullivan.

\begin{thm}
\label{thm-Sullivan}
A subspace of a smooth Banach space $X$ can be the range of at most one projection of norm one. 
\end{thm}

Now, we can deduce the following result.

\begin{prop}
\label{prop:range-pc}
Let $\cal{M}$ be a von Neumann algebra. Suppose that $1 < p < \infty$. Let $P \co \L^p(\cal{M}) \to \L^p(\cal{M})$ be a contractive projection. Then for any $h \in \L^p(\cal{M})$, we have
$$
P(x) 
= P(s_\ell(P) x s_r(P)), \quad x \in \L^p(\cal{M}).
$$
\end{prop}

\begin{proof}
The linear map $Q \co \L^p(\cal{M}) \to \L^p(\cal{M})$, $x \mapsto P(s_\ell(P) x s_r(P))$ is obviously a contractive projection on the subspace $\Ran P$. It is well-known that the Banach space $\L^p(\cal{M})$ is smooth by \cite[Corollary 5.2]{PiX03}. By Theorem \ref{thm-Sullivan}, we conclude that $P=Q$.
\end{proof}

\begin{remark}
\label{cas-L1} \normalfont
It is well-known that even in the commutative case and for positive contractive projections the statement of Proposition \ref{prop:range-pc} becomes false for $p=1$.
\end{remark}

%%%%%%%%%%%%%%%%%%%%%%%%%%%%%%%%%%%%%%%%%%%%%%%%%%%%%%%%%%%%%%%%%%%%%%%%%%%%%%%%%%%%%%%%%%%%%%%%
\section{A local description of 2-positive contractive projections}
\label{Sec-a-local}
%%%%%%%%%%%%%%%%%%%%%%%%%%%%%%%%%%%%%%%%%%%%%%%%%%

Let $P \co \L^p(\cal{M}) \to \L^p(\cal{M})$ be a non-trivial 2-positive contractive projection on an arbitrary noncommutative $\L^p$-space defined by a normal semifinite faithful weight $\varphi$. Consider a non-zero positive element $h$ of $\Ran P$. We have $P(h)=h$ and $h \in \L^p(\cal{M})_+$. We also consider the reduced von Neumann algebra $\cal{M}_h \ov{\mathrm{def}}{=} s(h)\cal{M}s(h)$. Recall that we have seen in Section \ref{Haagerup-noncommutative} that the space $s(h)\L^p(\cal{M})s(h)$ may be identified with the noncommutative $\L^p$-space $\L^p(\cal{M}_h)$ of the reduced von Neumann algebra $\cal{M}_h=s(h)\cal{M}s(h)$.

By applying Theorem \ref{Th-relevement-cp} to the restriction $P|_{s(h)\L^p(\cal{M})s(h)} \co s(h)\L^p(\cal{M})s(h) \to \L^p(\cal{M})$ and to the positive element $h$ of $\L^p(\cal{M})$, we see that there exists a unique linear map $\E \co \cal{M}_h \to s(P(h)) \cal{M}s(P(h))=\cal{M}_h$ such that
\begin{equation}
\label{relevement-E_h}
\quad P\big(h^{\frac{1}{2}}xh^{\frac{1}{2}}\big)
=h^{\frac{1}{2}} \E(x) h^{\frac{1}{2}}, \qquad x \in \cal{M}_h.
\end{equation}
Moreover, this map $\E$ is unital, contractive, normal and 2-positive. We will show that $\E \co \cal{M}_h \to \cal{M}_h$ is a conditional expectation onto a von Neumann subalgebra $\cal{N}$ of $\cal{M}_h$. We will also use the notation $\cal{N}_h$ for the algebra $\cal{N}$, specially in Section \ref{Carac-Nh}.

\begin{lemma}
\label{Lemma-faithful}
The linear map $\E \co \cal{M}_h \to \cal{M}_h$ is a normal faithful conditional expectation onto a von Neumann subalgebra $\cal{N}$ of $\cal{M}_h$.
\end{lemma}

\begin{proof}
We start by showing that the map $\E$ is faithful. Recall that for any $1 \leq p <\infty$, the norm of the Banach space $\L^p(\cal{M})$ is strictly monotone\footnote{\thefootnote. Suppose that $0 \leq x \leq y$. If $x \not=y$ then we have $\norm{x}_p < \norm{y}_p$.}. Now, we will show that if $0 \leq k \leq h$ and $P(k)=0$ then $k=0$. Indeed, from 
$$
h
=P(h)
=P(h)-P(k)
=P(h-k)
$$ 
we deduce that $\norm{h}_p = \norm{P(h-k)}_p \leq \norm{h-k}_p$ by the fact that $P$ is contractive. Since $0 \leq h-k\leq h$ we infer that $\norm{h}_p=\norm{h-k}_p$ and finally $k=0$ by strict monotonicity of the $\L^p$-norm. 

It results at once that $\E$ is faithful. Indeed, if $x \in \cal{M}_h^+$ and $\E(x)=0$ we have
$$
P\big(h^{\frac{1}{2}}xh^{\frac{1}{2}}\big)
\ov{\eqref{relevement-E_h}}{=} h^{\frac{1}{2}} \E(x) h^{\frac{1}{2}}
=0.
$$
Since $h^{\frac{1}{2}}xh^{\frac{1}{2}} \leq \norm{x}_{\infty}h$ by \cite[1.6.9 p.~18]{Dix77} we see that $h^{\frac{1}{2}}xh^{\frac{1}{2}}=0$ by the first part of the proof. Since $\cal{M}_h=s(h)\cal{M}s(h)$, we conclude that $x=0$ by Lemma \ref{lemma2-GL}. 

For any $x \in \cal{M}_h$, we have 
$$
P\big(h^{\frac{1}{2}}xh^{\frac{1}{2}}\big)
=P^2\big(h^{\frac{1}{2}}xh^{\frac{1}{2}}\big)
\ov{\eqref{relevement-E_h}}{=} P\big(h^{\frac{1}{2}} \E(x) h^{\frac{1}{2}}\big)
\ov{\eqref{relevement-E_h}}{=} h^{\frac{1}{2}} \E^2(x) h^{\frac{1}{2}}.
$$
Using the uniqueness of $\E$ given by Theorem \ref{Th-relevement-cp}, we infer that $\E^2=\E$, i.e.~$\E$ is a projection. Since $\E$ is in addition a 2-positive contraction, we deduce by Proposition \ref{Proj-cp-VNA} that $\cal{N} \ov{\mathrm{def}}{=} \Ran \E$ is a von Neumann subalgebra of $\cal{M}_h$ and that $\E$ is a conditional expectation from $\cal{M}_h$ onto $\cal{N}$.
\end{proof}

With the construction \eqref{Extension-weight1}, we can assume that the centralizer of the normal semifinite faithful weight $\varphi$ contains the projection $s(h)$. Note that $h^{p}$ is a positive element of the space $\L^1(\cal{M})$. At present, we consider the normal faithful positive linear form $\psi$ on the reduced von Neumann algebra $\cal{M}_h$ defined by
\begin{equation}
\label{psih-def}
\psi(x)
\ov{\mathrm{def}}{=} \tr_\varphi(h^px), \quad x \in \cal{M}_h.
\end{equation}

%that $h^{p}$ belongs to $\L^1(\cal{M})$, thus the map $x\mapsto \tr(h^px)$ defines a normal positive linear form on $\cal{M}$ with support $s(h)$ (a normal state if $\norm{h}_p=1$). Let $\psi_h$ be its restriction to the algebra $\cal{M}_h$.

Now, we will prove that the conditional expectation $\E$ is $\psi$-invariant. This property will allows us to extend $\E$ on all the spaces $\L^p(\cal{M}_h)$ in a compatible way, as explained in Section \ref{Sec-conditional}. 

\begin{lemma} 
\label{Lemma-prservation-state}
We have $\psi \circ \E=\psi$.
\end{lemma}  

\begin{proof}
Let $p^*$ be the exponent conjugate to $p$. Consider first the case where $1 < p < \infty$. Using the contractive dual map $P^* \co \L^{p^*}(\cal{M}) \to \L^{p^*}(\cal{M})$, we see that
\begin{align*}
\norm{h}_p^p 
&\ov{\eqref{Def-norm-Lp}}{=} \tr(h^p) 
= \tr(h h^{p-1}) 
= \tr\big((P(h) h^{p-1}\big)
= \tr\big(h P^*(h^{p-1})\big) \\
&\leq \norm{h}_p \bnorm{P^*(h^{p-1})}_{p^*} 
\leq \norm{h}_p \norm{h^{p-1}}_{p^*} = \norm{h}_p^p.
\end{align*}
Thus $\norm{P^*(h^{p-1})}_{p^*} = \norm{h^{p-1}}_{p^*}$ and $\tr \big(h P^* (h^{p-1})\big) = \norm{h}_p \norm{h^{p-1}}_{p^*}$. By \cite[Corollary 5.2]{PiX03}, the Banach space $\L^p(\cal{M}_h)$ is smooth. Hence, by uniqueness of the norming functional of $h$ (see \cite[Corollary 5.4.3 p.~481]{Meg98}) we obtain 
\begin{equation}
\label{equa-inter-400}
P^*(h^{p-1})
=h^{p-1}.
\end{equation}
For any $k \in \L^p(\cal{M})$, it follows that
\begin{equation}
\label{equa-inter-401}
\tr\big(h^{p-1}P(k)\big)
=\tr\big(P^*(h^{p-1})k\big)
\ov{\eqref{equa-inter-400}}{=}\tr(h^{p-1} k).
\end{equation}
In particular, for every $x \in \cal{M}_h$, we have
$$ 
\tr\big(h^{p-1}(h^{\frac{1}{2}}\E(x) h^{\frac{1}{2}})\big)
\ov{\eqref{relevement-E_h}}{=} \tr\big(h^{p-1}P(h^{\frac{1}{2}} x h^{\frac{1}{2}})\big)
\ov{\eqref{equa-inter-401}}{=} \tr\big(h^{p-1}(h^{\frac{1}{2}} x h^{\frac{1}{2}})\big)
$$
that is $\tr\big(h^p\E(x)\big)=\tr(h^p x)$ hence $\psi\big(\E(x)\big)=\psi(x)$.

%using the notation $u_h=s(h) P'(s(h)) s(h)$
In the case $p=1$, using the contractivity of $P^* \co \cal{M} \to \cal{M}$, we note that by \cite[1.6.9 p.~18]{Dix77}
$$
0 \leq P^*(s(h)) 
\leq \bnorm{P^*(s(h))}_\infty 1 
\leq \norm{s(h)}_\infty \leq  1.
$$
Hence $0 \leq s(h)P^*(s(h))s(h) \leq s(h)$. It follows that
\begin{equation}
\label{ineq:L1-case}
h^{\frac{1}{2}}P^*(s(h))h^{\frac{1}{2}}
= h^{\frac{1}{2}}s(h)P^*(s(h))s(h)h^{\frac{1}{2}}
\leq h^{\frac{1}{2}}s(h)h^{\frac{1}{2}}
=h
\end{equation}
However
$$ 
\tr\big(h^{\frac{1}{2}}P^*(s(h))h^{\frac{1}{2}}\big) 
= \tr(h P^*(s(h))) = \tr(P(h) s(h)) = \tr(h)
$$
hence
$$
\norm{h-h^{\frac{1}{2}}P^*(s(h))h^{\frac{1}{2}}}_1
=\tr\big(h-h^{\frac{1}{2}}P^*(s(h))h^{\frac{1}{2}}\big)
=0
$$
and the inequality in \eqref{ineq:L1-case} is in fact an equality, which implies by Lemma \ref{lemma2-GL}
\begin{align}
\label{P*-invar}
s(h)P^*(s(h))s(h)
=s(h).
\end{align}
Then for every $k \in s(h)\L^1(M)s(h)$ we have
\begin{align*}
\tr(P(k)) 
&= \tr(s(h)P(k)) 
= \tr(P^*(s(h))k) 
= \tr(P^*(s(h))s(h)ks(h)) \\
&=\tr(s(h)P^*(s(h))s(h)k)
\ov{\eqref{P*-invar}}{=} \tr(s(h)k)
= \tr(k).
\end{align*}
If $k=h^{\frac{1}{2}}y h^{\frac{1}{2}}$, with $y\in \cal{M}_h$, we obtain
\begin{align*}
\psi(\E(y))
=\tr\big(h^{\frac{1}{2}}\E y \,h^{\frac{1}{2}}\big) 
= \tr(P(h^{\frac{1}{2}}y h^{\frac{1}{2}})) 
=\tr (h^{\frac{1}{2}}y h^{\frac{1}{2}}) 
=\tr(hy)
\ov{\eqref{psih-def}}{=} \psi(y).
\end{align*}
We also deduce that $\psi \circ \E=\psi$ in this case.
\end{proof}

%Conversely, suppose that the conditions of Theorem \ref{main-th-ds-intro-bis} are satisfied. We can suppose $\norm{h}_p=1$. We introduce the reduced weight $\psi$ on the von Neumann algebra $\cal{M}_h \ov{\mathrm{def}}{=} s(h)\cal{M}s(h)$ induced by the state $\tr_{\varphi}(h^p\,\cdot)$.

In the sequel, we will use the density operator $D_{\psi} \in \L^1(\cal{M}_h,\psi)$ associated with the weight $\psi$ on $\cal{M}_h$. From \eqref{kappa}, we have a canonical map $\kappa$ which induces a complete order and isometric identification $\kappa \co \L^p(\cal{M}_h,\psi) \to \L^p(\cal{M}_h,\varphi_{s(h)})$ for all $p$. If $x \in \cal{M}_h$, using Lemma \ref{Lemme-trace coincide} in the third equality, we see that
\begin{equation}
\label{Magic-Formula}
\tr_{\psi}(D_{\psi} x) 
\ov{\eqref{HJX-1.13}}{=} \psi(x) 
\ov{\eqref{psih-def}}{=} \tr_{\varphi}(h^p x)
=\tr_{\varphi_{s(h)}}(h^p x)%\, \cdot
\ov{\eqref{Trace-preserving}}{=} \tr_{\psi}(\kappa^{-1}(h^p x))
=\tr_{\psi}(\kappa^{-1}(h)^p x).
\end{equation}
We conclude that
\begin{equation}
\label{Magic-formula-2}
D_{\psi}
=\kappa^{-1}(h)^p.
\end{equation}
With Lemma \ref{Lemma-prservation-state}, we can consider by \eqref{Map-extension-Lp} with $\cal{M}_h$ instead of $\cal{M}$ the completely positive contractive operator $\E_p \co \L^p(\cal{M}_h,\psi) \to \L^p(\cal{M}_h,\psi)$ induced by the conditional expectation $\E \co \cal{M}_h \to \cal{M}_h$ and defined by
\begin{equation}
\label{Def-esperance-h}
\E_p\big(D_{\psi}^{\frac{1}{2p}} x D_{\psi}^{\frac{1}{2p}}\big)
\ov{\eqref{Map-extension-Lp}}{=} D_{\psi}^{\frac{1}{2p}}\E(x)D_{\psi}^{\frac{1}{2p}},\quad x \in \cal{M}_h. 
\end{equation}
%For any $x \in \cal{M}_h$, note that
%\begin{equation*}
%\label{}
%Q_p^2\big(h_{\psi}^{\frac{1}{2p}} xh_{\psi}^{\frac{1}{2p}}\big)
%\ov{\eqref{Def-esperance-h}}{=} Q_p\big(h_{\psi}^{\frac{1}{2p}} Q(x)h_{\psi}^{\frac{1}{2p}}\big)
%\ov{\eqref{Def-esperance-h}}{=}h_{\psi}^{\frac{1}{2p}} Q^2(x)h_{\psi}^{\frac{1}{2p}}
%=h_{\psi}^{\frac{1}{2p}} Q(x)h_{\psi}^{\frac{1}{2p}}
%\ov{\eqref{Def-esperance-h}}{=} Q_p\big(h_{\psi}^{\frac{1}{2p}} xh_{\psi}^{\frac{1}{2p}}\big).
%\end{equation*}
%We deduce that $Q_p^2=Q_p$, i.e.~that the map $Q_p$ is a projection. 

For any $x \in \cal{M}_h$, we have
\begin{align*}
\MoveEqLeft
\kappa\E_p\kappa^{-1}(h^{\frac{1}{2}}xh^{\frac{1}{2}})       
=\kappa\E_p\big(\kappa^{-1}(h)^{\frac{1}{2}}x\kappa^{-1}(h)^{\frac{1}{2}}\big)  
\ov{\eqref{Magic-formula-2}}{=} \kappa\E_p\big(D_{\psi}^{\frac{1}{2p}}x\kappa^{-1}(h)D_{\psi}^{\frac{1}{2p}}\big)\\
&\ov{\eqref{Def-esperance-h}}{=} \kappa(D_{\psi}^{\frac{1}{2p}}\E(x)D_{\psi}^{\frac{1}{2p}}) 
\ov{\eqref{Magic-formula-2}}{=} h^{\frac{1}{2}}\E(x)h^{\frac{1}{2}}.
\end{align*}
By density with Lemma \ref{lemma2-GL}, we obtain
\begin{equation}
\label{sans-fin-3}
\kappa\E_p\kappa^{-1}
=P|_{s(h)\L^p(\cal{M})s(h)}.
\end{equation}
Since $s(h)$ belongs to the centralizer of $\varphi$, we have an order isometric identification of $\L^p(\cal{M}_h,\varphi_{s(h)})$ as a subspace of the Banach space $\L^p(\cal{M},\varphi)$. 
%\begin{align*}
%\MoveEqLeft
%P\kappa\big(h_{\psi}^{\frac{1}{2p}}xh_{\psi}^{\frac{1}{2p}}\big) 
%\ov{\eqref{Magic-formula-2}}{=} P\kappa\big(\kappa^{-1}(h)^{\frac{1}{2}}x \kappa^{-1}(h)^{\frac{1}{2}}\big)
%= P\big(h^{\frac{1}{2}}x h^{\frac{1}{2}}\big)
%\ov{\eqref{relevement-E_h-intro}}{=} h^{\frac{1}{2}}Q(x) h^{\frac{1}{2}} \\
%&\ov{\eqref{Magic-formula-2}}{=} \kappa(h_{\psi})^{\frac{1}{2p}}Q(x) \kappa(h_{\psi})^{\frac{1}{2p}} 
%=\kappa\big(h_{\psi}^{\frac{1}{2p}}Q(x) h_{\psi}^{\frac{1}{2p}}\big) \ov{\eqref{Def-esperance-h}}{=} \kappa Q_p\big(h_{\psi}^{\frac{1}{2p}}x h_{\psi}^{\frac{1}{2p}}\big).
%\end{align*}  
Note that the situation can be illustrated by following commutative diagram.
$$
\xymatrix @R=1cm @C=2cm{
\L^p(\cal{M},\varphi)\ar@{^{(}->}[r]^{P} &\L^p(\cal{M},\varphi)\\
\L^p\big(\cal{M}_h,\varphi_{s(h)}\big)\ar@{^{(}->}[u]  &  \L^p\big(\cal{M}_h,\varphi_{s(h)}\big)\ar@{^{(}->}[u]		\\
\L^p(\cal{M}_h,\psi)  \ar[u]^{\kappa}   \ar[r]^{\E_p} & \L^p(\cal{M}_h,\psi) \ar[u]_{\kappa} \\			
%     \L^p(M_h,\psi)  \ar[r]^{\E_{h,p}}\ar[u]^{\hat{\kappa}}&  \L^p(M_h,\psi) \ar[u]_{\hat{\kappa}}					
}
$$
Thus $P(\L^p(\cal{M}_h))$ is completely order isometrically isomorphic to the space $\L^p(\cal{N})$. Since $\L^p(\cal{N})$ is a $\cal{N}$-bimodule in $\L^p(\cal{M}_h)$, and the identification of $\L^p(\cal{M}_h)$ with $s(h)\L^p(\cal{M})s(h)$ respects the actions of $\cal{M}_h$, then $P(s(h)\L^p(\cal{M})s(h))$ is a $\cal{N}$-bimodule. Moreover, $P$ is $\cal{N}$-bimodular, that is 
$$
P(xky)
=xP(k)y, \quad  x,y \in \cal{N},\ k \in s(h)\L^p(\cal{M})s(h).
$$
This results simply from the $\cal{N}$-bimodularity of $\E_{p}$. Indeed, for any $x,y \in \cal{N}$ and any element $k \in s(h)\L^p(\cal{M})s(h)$, we have % by \cite[Prop. 2.3]{JuX03}.
\begin{align*}
\MoveEqLeft
P(xky)         
\ov{\eqref{sans-fin-3}}{=} \kappa\E_p\kappa^{-1}(xky) 
=\kappa\E_p(x\kappa^{-1}(k)y) 
\ov{\eqref{esp-cond-prop}}{=}\kappa (x\E_p(\kappa^{-1}(k))y) \\
&=x\kappa (\E_p(\kappa^{-1}(k)))y 
\ov{\eqref{sans-fin-3}}{=} xP(k)y.
\end{align*}

Note that the left supports $s_\ell(k)$ and $\sigma_\ell(k)$ of any element $k$ of $\L^p(\cal{N})$, viewed either as an element of $\L^p(\cal{M})$ or as an element of $\L^p(\cal{N})$, do coincide. Indeed we have clearly $\sigma_\ell(k) \geq s_\ell(k)$. Moreover if $r \ov{\mathrm{def}}{=}\sigma_\ell(k) - s_\ell(k)$, then $0 \leq \E(r) \leq \sigma_\ell(k)$ and $\E(r)k \ov{\eqref{esp-cond-prop}}{=} \E_{p}(rk)=0$, which imply $\E(r)=0$ and thus $r=0$ since $\E$ is faithful on $s(h)\cal{M}s(h)$. The same coincidence occurs of course for right supports.

Note that the projections of a $\sigma$-finite von Neumann algebra $\cal{M}$ are exactly the right (or left) supports of elements of $\L^p(\cal{M})$. Hence $\cal{N}$ is the von Neumann subalgebra of $\cal{M}$ generated by the right (resp. left) support projections of the elements of $P(s(h)\L^p(\cal{M})s(h))$. In particular $\cal{N}_h=\cal{N}$ only depends on $s(h)$, i.e.~$s(h_1)=s(h_2)$ implies $\cal{N}_{h_1}=\cal{N}_{h_2}$. 

\paragraph{The $\sigma$-finite case}
Here, we suppose that $\varphi$ is a normal faithful state on a $\sigma$-finite von Neumann algebra $\cal{M}$. With Theorem \ref{Prop-sh=sP}, we can consider a non-zero positive element $h$ of $\Ran P$ such that $s(P)=s(h)$. Using results of Section \ref{From-local-to-global}, we have
$$
\Ran P 
\ov{\eqref{Inter-200}}{=} P(s(h)\L^p(\cal{M})s(h))
=\kappa\E_p\kappa^{-1}(\L^p(\cal{M}_h)).
$$
 The linear map $\Phi \co \L^p(\cal{M}) \to \L^p(\cal{M}_h,\varphi_{s(h)})$, $y \mapsto s(h)ys(h)$ is completely positive. Using the last point of Theorem \ref{main-th-ds-intro} in the second equality, we obtain if $1 < p <\infty$ that
$$
\kappa \E_p\kappa^{-1}\Phi
\ov{\eqref{sans-fin-3}}{=} P|_{s(h)\L^p(\cal{M})s(h)}\Phi
=P
$$

%Note that the map $s(h)\cal{M}s(h) \to \C$, $s(h)xs(h) \mapsto \tr_{\varphi}(h^p x)$ is a faithful normal state on $s(h)\cal{M}s(h)$. Using the procedure \eqref{Extension-weight1}, we can consider a normal faithful state $\chi_h$ on $\cal{M}$ such that $s(h)$ belongs to the centralizer of $\chi_h$ and  the reduced weight $\psi_h \ov{\mathrm{def}}{=} (\chi_h)_{s(h)}$ on $\cal{M}_h \ov{\mathrm{def}}{=}s(h)\cal{M}s(h)$ satisfies
%\begin{equation}
%\label{}
%\psi_h(s(h)xs(h))
%=\tr_{\varphi}(h^p x), \quad x \in \cal{M}.
%\end{equation}
%If $\hat{\kappa} \co \L^1(\cal{M},\varphi) \to \L^1(\cal{M},\chi)$ is the canonical map, we obtain by \eqref{Trace-preserving} the equality
%\begin{equation}
%\label{}
%\psi_h(s(h)xs(h))
%=\tr_{\chi}(\hat{\kappa}(h^p x))
%=\tr_{\chi}(\hat{\kappa}(h)^p x),\quad x \in \cal{M}.
%\end{equation}
%By the results of Subsection \ref{Sec-a-local}, the range $\Ran P$ identifies to the noncommutative $\L^p$-spaces $\L^p(\cal{N}_h)$ of the von Neumann subalgebra $\cal{N}_{h}$ of $s(h)\cal{M}s(h)=s(P)\cal{M}s(P)$ generated by the right (resp. left) supports of the elements of $\Ran P$. 
%
%Moreover, the restriction of $P$ to $s(P)\L^p(\cal{M})s(P)$ identifies to a faithful normal conditional expectation in this reduced $\L^p$-space, with range $\L^p(\cal{N}_h)$. 

%%%%%%%%%%%%%%%%%%%%%%%%%%%%%%%%%%%%%%%%%%%%%%%%%%%%%%%%%%%%%%%%%%%%%%%%%%%%%%%%%%%%%%%
%\subsection{The non $\sigma$-finite case.}
%\label{The-non-sigma}
%%%%%%%%%%%%%%%%%%%%%%%%%%%%%%%%%%%%%%%%%%%%%%%%%%%%%%%%%%%%%%%%%%%%%%%%%%%%%%%%%%%%%%%

\section{The non-$\sigma$-finite case: the set of supports of elements in the range of $P$}
\label{Carac-Nh}

Let $\mathcal P(P)$ be the set of all support projections of positive elements in $\Ran P$. If $s \in \cal{P}(P)$ we  define $\ul_s$ as the von Neumann subalgebra of $s\cal{M}s$ generated by the projections $e \in \cal{P}(P)$ with $e\leq s$. The link with the subalgebras $\cal{N}=\cal{N}_h$ defined in section \ref{Sec-a-local} is given by the following lemma.

\begin{lemma}
\label{N-equality}
For every $h \in \Ran P$ we have $\ul_{s(h)}=\cal{N}_h$. 
\end{lemma}

\begin{proof}
Let $e \in \cal{N}_h$ be a projection. We have $e \leq s(h)$, hence the support $s(ehe)$ of $ehe$ is equal to $e$. Since $P(s(h)\L^p(\cal{M})s(h))$ is $\cal{N}_h$-bimodular, we have $P(ehe)=eP(h)e=ehe$.  So we have $ehe \in \Ran P$. Thus $e \in \cal{P}(P)$ and since $e \leq s(h)$ we see that $e \in \ul_{s(h)}$. Since the von Neumann algebra $\cal{N}_h$ is generated by its projections, it follows that $\cal{N}_h \subset \ul_{s(h)}$. 

Conversely if  $k \in \Ran P$ is positive with $s(k) \leq s(h)$, then $k=P(k)= P(s(h)ks(h))$, hence $k \in P(s(h)\L^p(\cal{M})s(h))$ and thus $s(k) \in \cal{N}_h$. Since $\ul_{s(h)}$ is generated by such projections, it follows that $\ul_{s(h)} \subset \cal{N}_h$.
\end{proof}

\begin{lemma}
If $s_1,s_2 \in \cal{P}(P)$ then the projection $s_1 \vee s_2$ belongs to $\cal{P}(P)$. 
\end{lemma}

\begin{proof}
We can write $s_i=s(h_i)$ for some positive element $h_i$ of $\Ran P$. Then the positive element $h=h_1+h_2$ belongs to $\Ran P$ and $s(h_1+h_2)= s(h_1) \vee s(h_2)$. 
\end{proof}

We denote by $\cal{N}_P$ the weak* closed $*$-algebra generated by $\cal{P}(P)$.

\begin{lemma}
\label{proj-in-Ns}
i) If $s \in \cal{P}(P)$ and $e \in \ul_s$ is a projection, then $e \in \cal{P}(P)$.\hfill\break
ii) if $s_1,s_2 \in \cal{P}(P)$ with $s_2 \leq s_1$ then $s_2 \in \ul_{s_1}$ and $\ul_{s_2}=s_2\ul_{s_1}s_2$.
\hfill\break
iii) for not necessarily comparable $s_1,s_2 \in \cal{P}(P)$ we have still $s_2\ul_{s_1}s_2 \subset \ul_{s_2}$.
\end{lemma}

\begin{proof}
i) If $e \in \ul_s$ then $e \in \cal{N}_h$ for some $h$ with $s(h)=s$. The first part of the proof of Lemma \ref{N-equality} gives the conclusion.
%(clearly $s(ehe)\le e$; on the other hand $e-s(ehe)$ is a projection $r\le e$ such that $rhr=0$; since $|h^{1/2}r |^2=rhr=0$ we have $h^{1/2}r=0$ and then  $hr=h^{1/2}h^{1/2}r =0$ ; in view of $r\le s(h)$ this implies $r=0$). 

\smallskip\noindent
ii) It follows from the definition of the algebras $\ul_s$ that $\ul_{s_2} \subset \ul_{s_1}$ whenever $s_2\leq s_1$. Then clearly $\ul_{s_2}=s_2\ul_{s_2}s_2 \subset s_2\ul_{s_1}s_2$.  Conversely every projection $e$ of the reduced von Neumann algebra $s_2\ul_{s_1}s_2$ belongs to $\ul_{s_1}$, hence to $\cal{P}(P)$ by (i), and is majorized by $s_2$, thus $e \in \ul_{s_2}$. It follows that $s_2\ul_{s_1}s_2 \subset \ul_{s_2}$.

\smallskip\noindent
iii) In this case we have $s_2\ul_{s_1}s_2 \subset s_2\ul_{s_1\vee s_2}s_2=\ul_{s_2}$.
\end{proof}
%the family $(\ul_s)_{s \in \cal{P}(P)}$ is nested by inclusion, so that

It follows from the two preceding lemmas that $\cal{N}_P=\ovl{\bigcup\limits_{s \in \cal{P}(P)} \ul_s}^{\w^*}$. Here is a more tractable definition. Recall that $s(P) \ov{\eqref{support-s(P)}}{=} \vee_{h \in \Ran P, h \geq 0} s(h)$ is the supremum of all projections in $\cal{P}(P)$. 

\begin{lemma}
The algebra $\cal{N}_P$ is the subset of $s(P)\cal{M}s(P)$ consisting of elements $x$ such that $sxs \in \ul_s$ for every $s \in \cal{P}(P)$.
\end{lemma}

\begin{proof}
Since each algebra $\ul_s$ is included in $s(P)\cal{M}s(P)$ and the latter algebra is weak* closed in $\cal{M}$, the whole algebra $\cal{N}_P$ is also included in $s(P)\cal{M}s(P)$. The set 
$$
A 
\ov{\mathrm{def}}{=} \big\{x \in s(P)\cal{M}s(P): sxs \in  \ul_s \hbox{ for all } s \in \cal{P}(P)\big\}
$$ 
is weak* closed and contains all the $\ul_s$'s with $s \in \cal{P}(P)$, thus $A$ contains $\cal{N}_P$. For every $x \in A$ and $s \in \cal{P}(P)$, we have $sxs \in \ul_s \subset \cal{N}_P$. Then $sxs \to s(P)xs(P)=x$ for the weak* topology when $s \uparrow s(P)$. 
%(as a consequence of the fact that $s\varphi \to s(P)\varphi$ in the norm of $\cal{M}_*$, for every $\varphi \in \cal{M}_*$)
 Thus $x \in \cal{N}_P$, and $A \subset \cal{N}_P$.
\end{proof}

\begin{lemma}
The map $P$ is $\cal{N}_P$-bimodular on $s(P)\L^p(\cal{M})s(P)$, and consequently $\Ran P$ is a $\cal{N}_P$-bimodule.
\end{lemma}

\begin{proof} 
We have to prove that $P(x_1hx_2)=x_1P(h) x_2$ for every $h \in s(P)\L^p(\cal{M})s(P)$ and $x_1,x_2 \in \cal{N}_P$. Clearly, we may reduce to the case where $x_1,x_2$ belong to some $\ul_{s_0}$ with $s_0 \in \cal{P}(P)$. Since $s(P)=\bigvee \cal{P}(P)$ we have 
$$
s(P) \L^p(\cal{M}) s(P)
=\ovl{\bigcup_{e\in \cal{P}(P)} e\L^p(\cal{M})e}
$$ 
(norm closure). So by another approximation argument we may assume that $h \in e\L^p(\cal{M})e$ with $e \in \mathcal P(P)$. Then, setting $s \ov{\mathrm{def}}{=} s_0 \vee e$, we have $x_1,x_2 \in \ul_s$ and $h\in s\L^p(\cal{M})s$. Since $P$ is bimodular with respect to $\ul_s$ on $s\L^p(\cal{M})s$, the equation $P(x_1hx_2)=x_1P(h) x_2$ follows, as well as the bimodule property. 
\end{proof}

Our goal is now to prove that $\Ran P$ is a suitable copy of $\L^p(\cal{N}_P)$ inside $s(P)\L^p(\cal{M})s(P)$. To this end we define a conditional expectation $\E_P$ from $s(P)\cal{M}s(P)$ onto $\cal{N}_P$ and an $\E_P$-invariant normal semifinite faithful weight on $s(P)\cal{M}s(P)$. A preliminary step will be to find a partition of $s(P)$ into projections belonging to $\mathcal P(P)$.

\begin{lemma}
\label{the lattice P(P)}
The set $\cal{P}(P)$ is closed under finite or countable unions, under arbitrary intersection and under relative orthocomplements.
\end{lemma}

\begin{proof}
We have seen that if $(h_i)_{i \in I}$ is an at most countable family of normalized positive elements in $\Ran P$, then $\bigvee_{i \in I} s(h_i)$ is the support of the element $h=\sum_i 2^{-i} h_i$ of $\Ran P$. Thus $\cal{P}(P)$ is closed under finite or countable joins.

If $s,e \in \mathcal P(P)$ with $e \leq s$, then $e$ belongs to $\ul_s$ by the part 2 of Lemma \ref{proj-in-Ns} and so does its relative orthocomplement $s-e$. Since $s-e$ is a projection in $\ul_s$ it  belongs to $\cal{P}(P)$, by the part 1 of Lemma \ref{proj-in-Ns}. 

If $s_1,s_2 \in \cal{P}(P)$, then $e_1 \ov{\mathrm{def}}{=} s_1\vee s_2 -s_1$ and $e_2 \ov{\mathrm{def}}{=} s_1\vee s_2 -s_2$ belong both to $\cal{P}(P)$ the foregoing, and so does $s_1\wedge s_2=s_1\vee s_2- e_1\vee e_2$.

If $(s_i)_{i \in I}$ is an arbitrary family in $\cal{P}(P)$ and $e \ov{\mathrm{def}}{=}\bigwedge_{i \in I} s_i$, then fixing $i_0 \in I$ we have $e = \bigwedge_{i \in I} s'_i$, where $s'_i=s_i\wedge s_{i_0}$. The $s'_i$ belong all to $\cal{P}(P)$, and in addition $s'_i \leq s_{i_0}$ with $s_{i_0} \in \cal{P}(P)$. By the part 2 of Lemma \ref{proj-in-Ns}, we have $s'_i \in \ul_{s_0}$ for every $i \in I$, hence $e \in \ul_{s_0}$ and this projection belongs to $\cal{P}(P)$ by the same lemma.
\end{proof}

\begin{thm}
\label{Thm-Disjoint}
There exists a family $(s_i)_{i \in I}$ of pairwise disjoint projections in $\mathcal P(P)$ such that 
$$
s(P)
=\bigvee_{i \in I} s_i.
$$
\end{thm}

\begin{proof}
Consider a maximal family $(s_i)_{i \in I}$ of pairwise disjoint non-vanishing projections in  $\mathcal P(P)$. Clearly $s\ov{\mathrm{def}}{=} \bigvee_{i \in I} s_i \leq s(P)$. If $s \ne s(P)$ there exists $e \in \cal{P}(P)$ such that $e\neq e\wedge s$. However, by Lemma \ref{the lattice P(P)}, the projection $e-e\wedge s=\bigwedge\limits_{i \in I}(e-e\wedge s_i)$ belongs to the set $\cal{P}(P)$, which contradicts the maximality of the family $(s_i)_{i \in I}$.
\end{proof}

%%%%%%%%%%%%%%%%%%%%%%%%%%%%%%%%%%%%%%%%%%%%%%%%%%%%%%%%%%%%%%%%%%%%%%%%%%%%%%%%%%
\section{The non-$\sigma$-finite case: proof of the main theorem}
\label{End-of-the-end}

In this section, we complete the proof of the points 1 and 2 of Theorem \ref{main-th-ds-intro} in the non-$\sigma$-finite case.

Fix a family $(s_i)_{i \in I}$ of pairwise disjoint projections of the set $\mathcal P(P)$ as in Theorem \ref{Thm-Disjoint}. For each $i \in I$, let $h_i$ be a positive element of $\Ran P$ with support projection $s(h_i)=s_i$. Let $\psi_i\ov{\mathrm{def}}{=} \tr(h_i^p\, \cdot)$ be the normal positive linear form on the von Neumann algebra $\cal{M}$ associated with the positive element $h_i^p$ of $\L^1(\cal{M},\varphi)$. For each finite subset $J$ of $I$, let us set $s_J \ov{\mathrm{def}}{=}\sum_{i \in J} s_i$, $h_J \ov{\mathrm{def}}{=} \sum_{i \in J} h_i$ and let $\psi_J$ be the linear form associated with $h_J^p$.  Note that $s(h_J)=s_J$. Moreover, we have 
$$
P(h_J)
=P\bigg(\sum_{i \in J} h_i\bigg)
=\sum_{i \in J} P(h_i)=h_J
\quad\text{and} \quad 
h_J^p
=\sum_{i \in J} h_i^p.
$$
Hence $
\psi_J = \sum_{i \in J} \psi_i$. Let us also define a normal semifinite faithful weight $\psi$ on the von Neumann algebra $s(P)\cal{M}s(P)$ by
\begin{align}
\label{global-weight}
\psi(x) 
\ov{\mathrm{def}}{=} \sum_{i \in I} \psi_i(x)
=\sum_{i \in I} \psi_i(s_i x s_i).
\end{align}
All the projections $s_i$ for $i \in I$ belong to the centralizer of $\psi$, and more generally so do the projections $s_\Gamma  \ov{\mathrm{def}}{=} w^*\sdash\sum_{i \in \Gamma} s_i$ where $\Gamma$ is a subset of $I$.

By the results of Section \ref{Sec-a-local}, we have a normal faithful conditional expectation $\E_J \ov{\mathrm{def}}{=} \E_{h_J} \co s_J\cal{M}s_J \to s_J\cal{M}s_J$ satisfying
\begin{equation}
\label{relevement-final}
\psi_J
=\psi_J\circ \E_J
\quad \text{and} \quad %\hat{\kappa}_JP\hat{\kappa}_J^{-1}|\L^p(s_JMs_J) 
P\big(h_J^{\frac{1}{2}}xh_J^{\frac{1}{2}}\big)
\ov{\eqref{relevement-E_h}}{=} h_J^{\frac{1}{2}}\E_J (x)h_J^{\frac{1}{2}}.
\end{equation} 
Moreover, by Lemma \ref{N-equality}, the range of $\E_J$ does not depend on the choice of the $h_i$ and equals $\ul_{s_J}=s_J \cal{N}_P s_J$ (by definition of $\cal{N}_P$).

In the sequel we denote by $\mathcal F(I)$ the set of finite subsets of $I$. This set is naturally ordered by inclusion.
\begin{lemma}
\label{lem:E-compatibility}
The family $(\E_J)_{J \in\mathcal F(I)}$ of conditional expectations is compatible, that is if $J_1\subset J_2$ then $\E_{J_2}|_{s_{J_1}\cal{M}s_{J_1}}=\E_{J_1}$. Moreover, we have $\psi_{J_2}|_{s_{J_1}\cal{M}s_{J_1}}=\psi_{J_1}$.
\end{lemma}

\begin{proof}
If $J_1\subset J_2$, then $h_{J_1}=s_{J_1}h_{J_2}=h_{J_2}s_{J_1}$, hence $h_{J_1}^{\frac{1}{2}}=s_{J_1}h_{J_2}^{\frac{1}{2}}=h_{J_2}^{\frac{1}{2}}s_{J_1}$ and for every $x \in s_{J_1} \cal{M}s_{J_1}$:
$$
h_{J_2}^{\frac{1}{2}}\E_{J_1}(x)h_{J_2}^{\frac{1}{2}}
=h_{J_1}^{\frac{1}{2}}\E_{J_1}(x)h_{J_1}^{\frac{1}{2}}
\ov{\eqref{relevement-final}}{=} P\big(h_{J_1}^{\frac{1}{2}}xh_{J_1}^{\frac{1}{2}}\big)
=P\big(h_{J_2}^{\frac{1}{2}}xh_{J_2}^{\frac{1}{2}}\big)
\ov{\eqref{relevement-final}}{=} h_{J_2}^{\frac{1}{2}}\E_{J_2}(x)h_{J_2}^{\frac{1}{2}}
$$
which implies $\E_{J_1}(x)=\E_{J_2}(x)$ since $\E_{J_1}(x)$ and $\E_{J_2}(x)$ belong to $s_{J_2}\L^p(\cal{M})s_{J_2}$. Hence the map $\E_{J_1}$ is the restriction of $\E_{J_2}$ to $s_1 \cal{M}s_1$. The other assertion is similar.
\end{proof}

Now, we want to extend the family $(\E_J)$ to a conditional expectation in $s(P)\cal{M}s(P)$. To this end we shall make use of the following lemma.

\begin{lemma}
\label{w*cv}
A bounded net $(x_\alpha)_{\alpha \in A}$ of elements of $s(P)\cal{M}s(P)$ converges w* if and only if the reduced nets $(s_J x_\alpha s_J)_\alpha$ where $J \in \cal{F}(I)$ are all convergent for the weak* topology.
\end{lemma}

\begin{proof} Note first that the maps $P_J \co s(P) \cal{M}s(P)\to s(P)\cal{M}s(P)$, $x \mapsto s_Jxs_J$ are weak* continuous projections on $s(P)\cal{M}s(P)$, which commute, i.e.~$P_JP_K=P_KP_J=P_{J\cap K}$, with the convention $P_\emptyset =0$. Moreover $P_J(x) \to x$ for the weak* topology when $J \uparrow I$, which implies that the family $(P_J)_{J \in \cal{F}(I)}$ is separating ($P_J(x)=0$ for all $J \in \cal{F}(I)$ implies $x=0$).

If the net $(x_\alpha)_{\alpha\in A}$ converges for the weak* topology to $x$, it is clear by the continuity of $P_J$ that $(P_J(x_\alpha))_{\alpha}$ converges for the weak* topology to $P_J(x)$ for each $J \in \mathcal F(I)$. Conversely, assume that $P_J(x_\alpha) \to y_J $ for the weak* topology for every $J \in \cal{F}(I))$. Then if $y$ is a weak*-cluster-point of the net $(x_\alpha)$ (some exists by weak*-compactness of the balls), $P_J(y)$ is a weak*-cluster point of $(P_J(x_\alpha))$, hence coincides with the weak*-limit $ y_J$ of the net $(P_J(x_\alpha))_\alpha$. By the fact that the family $(P_J)$ is separating, the weak*-cluster point $y$ is unique, and thus it is the weak*-limit of the net $(x_\alpha)$. 
\end{proof}

\begin{remark} 
\label{rem:local-def} \normalfont
As a consequence of the preceding lemma, for every bounded family  $(x_J)_{J \in \cal{F}(I)}$ such that $x_J \in s_J\cal{M}s_J$, and  $x_K= P_K(x_J)$ for every $K,J \in \cal{F}(I)$ with $K \subset J$, then there is a unique $x \in s(P)\cal{M}s(P)$ such that $x_K=P_K(x)$ for all $J \in \cal{F}(I)$.
\end{remark}

\begin{proof}
Indeed apply the Lemma \ref{w*cv} to the nested family $(x_J)_{J \in \mathcal F(I)}$. Since for each $K\in\mathcal F(I)$ the net $(P_K(x_J))_{J \in \mathcal F(I)}$ is stationary (constant for $J \supset K$), the net $(x_J)$ is w*-convergent and its limit $x$ satisfies the relations $P_K(x)=w^*\sdash\lim\limits_J P_K(x)_J=x_K$, for all $K \in \mathcal F(I)$.
\end{proof}

\begin{lemma}
\label{global-expect}
 The compatible set $(\E_J)$ of maps has a unique normal extension $\E \co s(P)\cal{M}s(P) \to s(P)\cal{M}s(P)$. The map $\E$ is a normal faithful conditional expectation with range $\cal{N}_P$.    
\end{lemma}

\begin{proof} {\it Existence of $\E$}. Note that if $K,J \in \cal{F}(I)$ satisfy $K \subset J$, we have
\begin{align}
\label{rel:E}
P_K\E_JP_J(x) 
= \E_JP_K(x) 
= \E_KP_K(x)
\end{align}
for all $x \in s(P)\cal{M}s(P)$ (the first equation by $\cal{N}_J$-bimodularity of $\E_J$, the second by Lemma \ref{lem:E-compatibility}). By Remark \ref{rem:local-def} there exists a unique map $\E \co s(P)\cal{M}s(P) \to s(P)\cal{M}s(P)$ such that:
\begin{align}
\label{def-E}
P_K\E(x)
= \E_KP_K(x)
\end{align}
for all $x \in s(P) \cal{M}s(P)$. Moreover we have $\E(x) =w^* \sdash \lim_J \E_JP_J(x)$. If $x \in s_K\cal{M}s_K$ we have  $\E_J(x)= \E_K(x)$ for all $J \supset K$ and thus $\E(x) = \E_K(x)$, i.e.~$\E$ extends all the $\E_K$, $K \in \cal{F}(I)$. Then by Lemma \ref{w*cv} the weak* continuity of $\E$ follows from that of the $\E_J$ where $J \in \cal{F}(I)$. As for the faithfulness, it results from \eqref{rel:E} and faithfulness of the $\E_J$, since if $x \geq 0$ we have
$$
\E x
= 0 \implies \forall J \in \cal{F}(I),\ 
\E_JP_J(x)
= P_J\E(x)=0\implies \forall J,\ 
P_J(x)=0
\implies x=0.
$$
It is clear that $\E$ is $s_J\cal{M}s_J$-bimodular for every $J$. The $\cal{N}_P$-modularity follows by weak* continuity. Finally $\Ran \E$ is included in $\cal{N}_P$ and contains $s(P)$. By $\cal{N}_P$-modularity, $\Ran \E=\cal{N}_P$.

\medskip
\noindent{\it Uniqueness.}  If $\cal{E}$ is another extension of the $\E_J$'s then by weak* continuity
$$
\cal{E}(x)
= \cal{E}\big(w^*\sdash \lim_J P_Jx \big) 
= w^*\sdash\lim_J \cal{E} P_J(x)
=w^*\sdash\lim_J\E_JP_J(x) 
= w^*\sdash\lim_J\E P_J(x) 
=\E(x).
$$
\end{proof}

\begin{lemma}
 We have $\psi=\psi \circ \E$.
\end{lemma}

\begin{proof}
For $x\in s(P)\cal{M}_+s(P)$ we have
$$
\psi(\E(x))
\ov{\eqref{global-weight}}{=} \sum_{i \in I} \psi_i(s_i\E (x)s_i) 
=\sum_{i \in I} \psi_i(\E_i(s_ixs_i))
=\sum_{i \in I} \psi_i(s_ixs_i)
=\psi(x)
$$
\end{proof}

%As explained in Section \ref{Sec-conditional}, there is a canonical inclusion of the $\L^0$ space of the crossed product $\cal{N}_P \rtimes_{\sigma^\psi} \R$ into that of  $s(P)\cal{M}s(P)\rtimes_{\sigma^\psi} \R$, with preservation of their canonical traces, which leads to 

We can consider the inclusions $i_p \co \L^p(\cal{N}_P,\psi|_{\cal{N}_P}) \subset \L^p(s(P)\cal{M}s(P),\psi)$ for any $1 \leq p \leq \infty$. The $\L^p$-extension $\E_{p} \co \L^p(s(P)\cal{M}s(P), \psi) \to \L^p(s(P)\cal{M}s(P),\psi)$ is defined by transposition of $i_{p^*}$.%(for the Haagerup trace duality).

\begin{lemma}\label{comm-diag-1}
We have the following commutative diagram:
$$
 \xymatrix @R=1cm @C=2cm
{\L^p(s(P)\cal{M}s(P),\psi)\ar[r]^{\E_p}&\L^p(s(P)\cal{M}s(P),\psi)\\
 s_J\L^p(s(P)\cal{M}s(P),\psi)s_J=\L^p(s_J\cal{M}s_J) \ar@{^{(}->}[u]\ar[r]^{\E_{J,p}}   &  \L^p(s_J\cal{M}s_J) = s_J\L^p(s(P)\cal{M}s(P),\psi)s_J\ar@{^{(}->}[u]}
$$

\end{lemma}
\begin{proof}
For any $h \in \L^p(s_J\cal{M}s_J,\psi)$ and any $k \in \L^{p_*}(\cal{N}_P,\psi)$, we have using twice the fact that $s_J$ is in the centralizer of $\psi$
\begin{align*}
\tr_\psi(hk)
&= \tr_\psi (s_Jhs_Jk)
=\tr_\psi(hs_Jks_J)
=\tr_{\psi_J}(hs_Jks_J) \\ 
&\ov{\eqref{Ep}}{=} \tr_{\psi_J}(\E_{J,p} (h)s_Jks_J) 
= \tr_{\psi}(\E_{J,p} (h)s_Jks_J) \\ 
& = \tr_{\psi}(s_J\E_{J,p}(h)s_Jk) 
= \tr_{\psi}(\E_{J,p} (h)k).
\end{align*}
With the characterization of \eqref{Ep}, we conclude that $\E_{p}(h)= \E_{J,p}(h)$.
\end{proof}

Using an isomorphism $\kappa \co \L^p(s(P)\cal{M}s(P),\psi) \to \L^p(s(P)\cal{M}s(P),\varphi_{s(P)})$ provided by \eqref{kappa}, the following diagram summarizes the present situation.
$$
\xymatrix @R=1cm @C=2cm{
\L^p(\cal{M},\varphi)\ar@{^{(}->}[r]^{P} &\L^p(\cal{M},\varphi)\\
\L^p\big(s(P)\cal{M}s(P),\varphi_{s(P)}\big)\ar@{^{(}->}[u]  &  \L^p\big(s(P)\cal{M}s(P),\varphi_{s(P)}\big)\ar@{^{(}->}[u]		\\
 s(P)\L^p(\cal{M},\chi)s(P)=\L^p(s(P)\cal{M}s(P),\psi) \ar[u]^{\kappa}   \ar[r]^{\E_{p}} &  s(P)\L^p(\cal{M},\chi)s(P)=\L^p(s(P)\cal{M}s(P),\psi) \ar[u]_{\kappa} \\			
%     \L^p(M_h,\psi)  \ar[r]^{\E_{h,p}}\ar[u]^{\hat{\kappa}}&  \L^p(M_h,\psi) \ar[u]_{\hat{\kappa}}					
}
$$

%Let us consider now a normal semifinite faithful weight $\chi$ on $\cal{M}$ such that $s(P)$ belongs to the centralizer of $\chi$, and $\chi_{s(P)}=\chi|_{s(P)\cal{M}s(P)} = \psi$. Let $\kappa \co (\cal{M} \rtimes_{\sigma^\varphi_i} \R,\tau_\varphi)\to (\cal{M} \rtimes_{\sigma^\chi_i} \R,\tau_\chi)$ be the canonical trace-preserving isomorphism, $\hat\kappa \co \L^0(\cal{M} \rtimes_{\sigma^\varphi_i} \R,\tau_\varphi) \to \L^0(\cal{M} \rtimes_{\sigma^\chi_i} \R,\tau_\chi)$ its natural extension, which preserves the respective Haagerup's noncommutative $\L^p(\cal{M})$-spaces. Recall that if we identify $\cal{M}$ with its canonical image in each crossed product then by \eqref{kappa} $\hat \kappa$ fixes the points of $\cal{M}$.

%\begin{prop}\label{comm-diag-2}
%We have the following commutative diagram:
%$$
%\xymatrix @R=1cm @C=2cm{
%\L^p(\cal{M},\varphi)\ar@{^{(}->}[r]^{P} &\L^p(\cal{M},\varphi)\\
%\L^p(\cal{M},\chi)  \ar[u]^{\hat{\kappa}^{-1}}&  \L^p(\cal{M},\chi) \ar[u]_{\hat{\kappa}^{-1}}		\\
 %s(P)\L^p(\cal{M},\chi)s(P)=\L^p(s(P)\cal{M}s(P),\psi)  \ar@{^{(}->}[u]   \ar[r]^{\E_{\psi,p}}   &  \L^p(s(P)\cal{M}s(P),\psi)=s(P)\L^p(\cal{M},\chi)s(P)\ar@{^{(}->}[u]
 %}
%$$
%\end{prop}
We will prove that \textit{on the subspace} $s(P)\L^p(\cal{M})s(P)$ we have
\begin{equation}
\label{sans-fin-R}
\kappa\E_p\kappa^{-1}
=P.
\end{equation}
For every $J \in \cal{F}(I)$, it follows by \eqref{sans-fin-3} that
$$
\kappa\E_{J,p}\kappa^{-1}(k)
=P|_{s(P)\L^p(\cal{M})s(P)}(k)
$$
for all $k \in s_J\L^p(s(P)\cal{M}s(P),\psi)s_J = \L^p(s_J \cal{M} s_J,\psi_J)$. Since by Lemma \ref{comm-diag-1}, $\E_{J,p} $ is the restriction of $\E_{p} \co \L^p(s(P)\cal{M} s(P),\psi)  \to \L^p(s(P)\cal{M} s(P),\psi)$ to the subspace $s_J\L^p(s_J \cal{M} s_J,\psi_J)s_J$ we obtain
$$
\kappa\E_p\kappa^{-1}(k)
=P|_{s(P)\L^p(\cal{M})s(P)}(k).
$$
for all $k\in \bigcup_J \L^p(s_J\cal{M}s_J)$. By density of the subspace $\bigcup_J \L^p(s_J\cal{M}s_J)$ in the Banach space $\L^p(s(P)\cal{M}s(P))$ it follows that $P$ coincides with $\kappa\E_p\kappa^{-1}$ on $s(P)\L^p(\cal{M},\varphi)s(P)$.

%
%%%%%%%%%%%%%%%%%%%%%%%%%%%%%%%%%%%%%%%%%%%%%%%%%%%%%%%%%%%%%%%%%%%%%%%%%%%%%%%%%%%%%%%%%%%%%
%\subsection{Complement on the case $1 < p < \infty$ \textbf{OK}} 
%\label{sec:Complement}
%%%%%%%%%%%%%%%%%%%%%%%%%%%%%%%%%%%%%%%%%%%%%%%%%%%%%%%%%%%%%%%%%%%%%%%%%%%%%%%%%%%%%%%%%%%%%

\vspace{0.2cm}

\textbf{Declaration of competing interest}.
The authors have no conflicts of interest to declare that are relevant to the content
of this article.

\vspace{0.2cm}

\textbf{Data availability}.
Data sharing is not applicable to this article as no datasets were generated or analysed
during the current study.

\vspace{0.2cm}

\textbf{Acknowledgment}.
The author acknowledges support by the grant ANR-18-CE40-0021 (project HASCON) of the French National Research Agency ANR. Finally, the authors would like to thank Quanhua Xu for us providing an expanded version of his paper \cite{HJX10} from his \textit{own} initiative. Finally, we would like to express our gratitude to the referee for their contribution in simplifying the proof of Proposition \ref{prop:range-pc} (our original approach was based on a result of Calvert \cite{Cal75}), and for their careful and thorough review of our paper.

%The first named author would like to thank Universit\'e of Paris 6 where he had fruitful discussions with the second named author during his visit in 2015. 

%%%%%%%%%%%%%%%%%%%%%%%%%%%%%%%%%%%%%%%%%
%%%%%%%%%%%%%%%%%%%%%%%%%%%%%%%%%%%%%%%%%

{\footnotesize

\vspace{0.2cm}

\noindent C\'edric Arhancet\\ 
\noindent 6 rue Didier Daurat, 81000 Albi, France\\
URL: \href{http://sites.google.com/site/cedricarhancet}{https://sites.google.com/site/cedricarhancet}\\
cedric.arhancet@protonmail.com\\

\noindent Yves Raynaud\\ 
\n Sorbonne Universit\'e, Universit\'e Paris-Diderot, CNRS,\\
Institut de Math\'ematiques de Jussieu-Paris Rive Gauche, IMJ-PRG, \\
F-75005, Paris, France.\\
yves.raynaud@upmc.fr\hskip.3cm
}


\small
\begin{thebibliography}{79}

\bibitem[AbA02]{AbA02}
Y. A. Abramovich and C. D. Aliprantis.
\newblock An invitation to operator theory.
\newblock Graduate Studies in Mathematics, 50. American Mathematical Society, Providence, RI, 2002.

\bibitem[AcC82]{AcC82}
L. Accardi and C. Cecchini.
\newblock Conditional expectations in von Neumann algebras and a theorem of Takesaki.
\newblock J. Funct. Anal. 45 (1982), 245--273.

%\bibitem[AcL]{AcL}
%L. Accardi and R. Longo.
%\newblock Martingale convergence of generalized conditional expectations.
%\newblock J. Funct. Anal. 118 (1993), no. 1, 119--130.

\bibitem[AlT07]{AlT07}
C. D. Aliprantis and R. and Tourky.
\newblock Cones and Duality.
\newblock Graduate Studies in Mathematics, 84. American Mathematical Society, Providence, RI, 2007.

\bibitem[And66]{And66}
T. Ando.
\newblock Contractive projections in $L_{p}$ spaces.
\newblock Pacific J. Math. 17 (1966), 391--405.

%\bibitem[AnD]{AnD}
%C. Anantharaman-Delaroche.
%\newblock On ergodic theorems for free group actions on noncommutative spaces.
%\newblock Probab. Theory Related Fields 135 (2006), no. 4, 520--546.

%\bibitem[AHW]{AHW}
%H. Ando, U. Haagerup, C. Winslow.
%\newblock Ultraproducts, QWEP von Neumann Algebras, and the Effros-Maréchal Topology.
%\newblock Arxiv.



%\bibitem[AlS]{AlS}
%Alfsen-Shultz.
%\newblock Geometry of state spaces of operator algebras.
%\newblock Book.

%\bibitem[AnH]{AnH}
%H. Ando and U. Haagerup.
%\newblock Ultraproducts of von Neumann algebras.
%\newblock J. Funct. Anal. 266 (2014), no. 12, 6842--6913.  

\bibitem[ArF78]{ArF78}
J. Arazy and Y. Friedman.
\newblock Contractive projections in $C_1$ and $C_\infty$.
\newblock Mem. Amer. Math. Soc. 13 (1978), no. 200, 1--165.

\bibitem[ArF92]{ArF92}
J. Arazy and Y. Friedman.
\newblock Contractive projections in $C_p$.
\newblock Mem. Amer. Math. Soc. 459 (1992), 1--109.

%\bibitem[Arh19]{Arh1}
%C. Arhancet.
%\newblock Dilation of semigroups on von Neumann algebras and noncommutative $\mathrm{L}^p$-spaces.
%\newblock J. Funct. Anal. 276 (2019), no. 7, 2279--2314.

\bibitem[Arh20]{Arh20}
C. Arhancet.
\newblock Positive contractive projections on noncommutative $\L^p$-spaces and nonassociative $\L^p$-spaces.
\newblock Preprint, arXiv:1909.00391.

\bibitem[Arh23a]{Arh23a}
C. Arhancet.
\newblock Contractively decomposable projections on noncommutative $\L^p$-spaces.
\newblock J. Math. Anal. Appl. 533 (2024), no. 2, Paper No. 128017.

\bibitem[Arh23b]{Arh23b}
C. Arhancet.
\newblock Nonassociative $\L^p$-spaces and embeddings in noncommutative $\L^p$-spaces.
\newblock Preprint, arXiv:2307.04452.

\bibitem[ArK23]{ArK23}
C. Arhancet and C. Kriegler.
\newblock Projections, multipliers and decomposable maps on noncommutative $\L^p$-spaces.
\newblock M\'em. Soc. Math. Fr. (N.S.) (2023), no. 177.

\bibitem[BeL74]{BeL74}
S. J. Bernau and H. E. Lacey.
\newblock The range of a contractive projection on an $L^p$-space.
\newblock Pacific J. Math. 53 (1974), 21--41.

%\bibitem[BLM]{BLM}
%Blecher Le Merdy.
%\newblock Book.
%\newblock .


%\bibitem[Bla]{Bla}
%B. Blackadar.
%\newblock Operator algebras. Theory of $C^*$-algebras and von Neumann algebras.
%\newblock Encyclopaedia of Mathematical Sciences, 122. Operator Algebras and Non-commutative Geometry, III. Springer-Verlag, Berlin, 2006.

%\bibitem[BLM]{BLM}
%D. Blecher and C. Le Merdy.
%\newblock Operator algebras and their modules-an operator space approach.
%\newblock London Mathematical Society Monographs. New Series, 30. Oxford Science Publications. The Clarendon Press, Oxford University Press, Oxford, 2004.

%\bibitem[CL]{CL}
 %Valerio Capraro, Martino Lupini.
%\newblock Introduction to Sofic and Hyperlinear groups and Connes' embedding conjecture.
%\newblock .

\bibitem[Cal75]{Cal75}
B. Calvert.
\newblock Convergence sets in reflexive Banach Spaces.
\newblock Proc. Amer. Math. Soc. 47 (1975), no. 2, 423--428.

%\bibitem[ChE]{ChE}
%M. D. Choi and E. G. Effros.
%\newblock Injectivity and operator spaces.
%\newblock J. Funct. Anal. 24 (1977), no. 2, 156--209.

%\bibitem[Chd]{Chd}
%M. Choda.
%\newblock A Remark on the Normal Expectations.
%\newblock Proc. Japan Acad. 44 (1968), 462--466.

\bibitem[CoS70]{CoS70}
H. B. Cohen and F. E. Sullivan.
\newblock Projecting onto cycles in smooth, reflexive Banach spaces. 
\newblock Pacific J. Math. 34 (1970), 355--364. 

%\bibitem[Con]{Con1}
%A. Connes.
%\newblock Géométrie non commutative.
%\newblock .
%
\bibitem[Con73]{Con73}
A. Connes.
\newblock Une classification des facteurs de type III.
\newblock Ann. Sci. Ecole Norm. Sup. (4) 6 (1973), 133-252.

%\bibitem[CPPR]{CPPR}
%M. Caspers, J. Parcet, M. Perrin and \'E. Ricard.
%\newblock Noncommutative de Leeuw theorems.
%\newblock Forum Math. Sigma 3 (2015), e21.

%\bibitem[Dae1]{Dae1}
%A. van Daele.
%\newblock Continuous crossed products and type III von Neumann algebras.
%\newblock London Mathematical Society Lecture Note Series, 31. Cambridge University Press, Cambridge-New York, 1978. 

\bibitem[DeJ04]{DeJ04}
A. Defant and M. Junge.
\newblock Marius Maximal theorems of Menchoff-Rademacher type in non-commutative $L_q$-spaces.
\newblock J. Funct. Anal. 206 (2004), no. 2, 322--355. 

\bibitem[Dix77]{Dix77}
J. Dixmier.
\newblock $C\sp*$-algebras. Translated from the French by Francis Jellett.
\newblock North-Holland Mathematical Library, Vol. 15. North-Holland Publishing Co., Amsterdam-New York-Oxford, 1977. 

%\bibitem[Dix2]{Dix2}
%J. Dixmier.
%\newblock Von Neumann algebras. With a preface by E. C. Lance. Translated from the second French edition by F. Jellett
%\newblock North-Holland Mathematical Library, 27. North-Holland Publishing Co., Amsterdam-New York, 1981. 

%\bibitem[DP]{DP}
%B. de Pagter.
%\newblock Non-commutative Banach function spaces.
%\newblock Positivity, 197--227, Trends Math., Birkhäuser, Basel, 2007.

%\bibitem[DP]{DP}
%De Pagter.
%\newblock Course von Neumann algebras pdf.
%\newblock .


%\bibitem[DLG2]{DLG2}
%K. de Leeuw and I. Glicksberg.
%\newblock The decomposition of certain group representations.
%\newblock J. Analyse Math. 15 (1965), 135--192.

\bibitem[DHP90]{DHdP90}
P. G. Dodds, C. B. Huijsmans and B. de Pagter.
\newblock Characterization of conditional expectation-type operators.
\newblock Pacific J. Math. 14 (1990), no 1, 55--77.


\bibitem[Dou65]{Dou65}
R. G. Douglas.
\newblock Contractive projections on an $\mathcal{L}_{1}$ space.
\newblock Pacific J. Math. 15 (1965), 443--462.

%\bibitem[EFHN]{EFHN}
%T. Eisner, B. Farkas, M. Haase and R. Nagel.
%\newblock Operator theoretic aspects of ergodic theory.
%\newblock Graduate Texts in Mathematics, 272. Springer, Cham, 2015.

\bibitem[EfR00]{EfR00}
E. Effros and Z.-J. Ruan.
\newblock Operator spaces.
\newblock Oxford University Press (2000).

%\bibitem[EfS]{EfS}
%E. Effros and E. St\o rmer.
%\newblock Positive projections and Jordan structure in operator algebras.
%\newblock Math. Scand. 45 (1979), no. 1, 127--138.

%\bibitem[EfS2]{EfS2}
%E. Effros and E. St\o rmer.
%\newblock Jordan algebras of self-adjoint operators.
%\newblock Trans. Amer. Math. Soc. 127 1967 313--316.

\bibitem[FrB85]{FrB85}
Y. Friedman and B. Russo.
\newblock Solution of the contractive projection problem.
\newblock J. Funct. Anal. 60 (1985), no. 1, 56--79.

\bibitem[Gol85]{Gol85}
S. Goldstein.
\newblock Conditional expectations in $L^p$-spaces over von Neumann algebras.
\newblock  Quantum probability and applications, II (Heidelberg, 1984), 233--239, Lecture Notes in Math., 1136, Springer, Berlin, 1985.   

%\bibitem[Gol2]{Gol2}
%S. Goldstein.
%\newblock Conditional expectation and stochastic integrals in noncommutative $L^p$ spaces.
%\newblock Math. Proc. Cambridge Philos. Soc. 110 (1991), no. 2, 365--383.
	
\bibitem[GoL09]{GoLa09}
S. Goldstein and L. E. Labuschagne.
\newblock Composition operators on Haagerup $L^p$-spaces.
\newblock Infin. Dimens. Anal. Quantum Probab. Relat. Top. 12 (2009), no. 3, 439--468.	

%\bibitem[GuR]{GuR}
%S. Guerre. and Y. Raynaud.
%\newblock Sur les isométries de $L^p(X)$ et le théorème ergodique vectoriel. (French) [The isometries of $L^p(X)$ and the vector-valued ergodic theorem].
%\newblock Canad. J. Math. 40 (1988), no. 2, 360--391.

%\bibitem[Haa]{Haa}
%U. Haagerup.
%\newblock Injectivity and decomposition of completely bounded maps.
%\newblock .

%\bibitem[Han]{Han}
%Kyung Hoon Han.
%\newblock Noncommutative $L_p$-space and operator system.
%\newblock .

%\bibitem[HaM]{HaM}
%U. Haagerup and M. Musat.
%\newblock Factorization and dilation problems for completely positive maps on von Neumann algebras.
%\newblock Comm. Math. Phys. 303 (2011), no. 2, 555--594.

%\bibitem[HeR]{HeR}
 %C. W. Henson and Y. Raynaud.
%\newblock On the theory of $L_p(L_q)$-Banach lattices.
%\newblock  Positivity 11 (2007), no. 2, 201--230.


%\bibitem[Izu]{Izu}
%H. Izumi.
%\newblock Constructions of non-commutative $L^p$-spaces with a complex parameter arising from modular actions.
%\newblock Internat. J. Math. 8 (1997), no. 8, 1029--1066.


%\bibitem[Jun]{Jun}
%M. Junge.
%\newblock Fubini's theorem for ultraproducts of noncommmutative $L_p$-spaces.
%\newblock Canad. J. Math. 56 (2004), no. 5, 983--1021.

%\bibitem[Jun2]{Jun2}
%M. Junge.
%\newblock Fubini's theorem for ultraproducts of noncommmutative $L_p$-spaces II.
%\newblock Preprint.
%
%\bibitem[JR1]{JR1}
%M. Junge and Z.-J. Ruan.
%\newblock Approximation properties for noncommutative $L_p$-spaces associated with discrete groups.
%\newblock Duke Math. J. 117 (2003), no. 2, 313--341.


%\bibitem[JoL]{JoL}
%W. B. Johnson and J. Lindenstrauss.
%\newblock Basic concepts in the geometry of {B}anach spaces, Handbook of the geometry of {B}anach spaces, {V}ol. {I}.
%\newblock  North-Holland, Amsterdam, 2001, pp.~1--84.

%\bibitem[GL1]{GL1}
%S. Goldstein and J. M. Lindsay.
%\newblock KMS-symmetric Markov semigroups.
%\newblock Math. Z. 219 (1995), no. 4, 591--608.

%\bibitem[GL2]{GL2}
%S. Goldstein and J. M. Lindsay.
%\newblock Markov semigroups KMS-symmetric for a weight.
%\newblock Math. Ann. 313 (1999), no. 1, 39--67.





%\bibitem[Haa3]{Haa3}
%U. Haagerup.
%\newblock Injectivity and decomposition of completely bounded maps.
%\newblock Operator algebras and their connections with topology and ergodic theory (Bu\c{s}teni, 1983), 170--222, Lecture Notes in Math., 1132, Springer, Berlin, 1985.
 
%\bibitem[Haa78a]{Haa4}
%U. Haagerup.
%\newblock On the dual weights for crossed products of von Neumann algebras. I. Application of operator valued weights.
%\newblock Math. Scand. 43 (1978/79), no. 1, 99--118.

\bibitem[Haa78]{Haa78}
U. Haagerup.
\newblock On the dual weights for crossed products of von Neumann algebras. II. Removing separability conditions.
\newblock Math. Scand. 43 (1978/79), no. 1, 119--140.

\bibitem[Haa79a]{Haa79a}
U. Haagerup.
\newblock Operator-valued weights in von Neumann algebras. II.
\newblock J. Funct. Anal. 33 (1979), no. 3, 339--361.

\bibitem[Haa79b]{Haa79b}
U. Haagerup.
\newblock $L^p$-spaces associated with an arbitrary von Neumann algebra. Alg\`ebres d'op\'era-teurs et leurs applications en physique math\'ematique (Marseille, 1977), 175--184, Colloq. Internat. CNRS, 274, CNRS, Paris, 1979.

\bibitem[HJX10]{HJX10}
U. Haagerup, M. Junge and Q. Xu.
\newblock A reduction method for noncommutative $L_p$-spaces and applications.
\newblock Trans. Amer. Math. Soc. 362 (2010), no. 4, 2125--2165.

%\bibitem[HM]{HM}
%U. Haagerup and M. Musat.
%\newblock Factorization and dilation problems for completely positive maps on von Neumann algebras.
%\newblock Comm. Math. Phys. 303 (2011), no. 2, 555--594.

\bibitem[HRS03]{HRS03}
U. Haagerup, H.P. Rosenthal and F. A. Sukochev.
\newblock Banach embedding properties of non-commutative $L^p$-spaces.
\newblock Mem. Amer. Math. Soc. 163 (2003), no. 776.

%\bibitem[HS1]{HS1}
%U. Haagerup and E. St\o rmer.
%\newblock Positive projections of von Neumann algebras onto JW-algebras.
%\newblock Proceedings of the XXVII Symposium on Mathematical Physics (Torun, 1994). Rep. Math. Phys. 36 (1995), no. 2-3, 317--330.

\bibitem[HiT87]{HiT87}
F. Hiai  and M. Tsukada 
\newblock Generalized conditional expectations and martingales in non-commutative $L^p$-spaces. 
\newblock J. Operator Theory 18 (1987), no. 2, 265--288.

\bibitem[HoT09]{HoT09}
T. Honda and W. Takahashi.
\newblock Norm one linear projections and generalized conditional expectations in Banach spaces.
\newblock Sci. Math. Jpn. 69 (2009), no. 3, 303--313.

%\bibitem[HOT1]{HOT1}
%H. Hanche-Olsen and E. Størmer.
%\newblock Jordan operator algebras.
%\newblock Monographs and Studies in Mathematics, 21. Pitman (Advanced Publishing Program), Boston, MA, 1984.

%\bibitem[JR]{JR}
%M. Junge and Z.-J. Ruan.
%\newblock Decomposable maps on non-commutative $L_p$-spaces.
%\newblock Operator algebras, quantization, and noncommutative geometry, 355--381, Contemp. Math., 365, Amer. Math. Soc., Providence, RI, 2004.

\bibitem[JRX05]{JRX05}
M. Junge, Z.-J. Ruan and Q. Xu.
\newblock Rigid $\mathscr{OL}_p$ structures of non-commutative $L^p$-spaces associated with hyperfinite von Neumann algebras.
\newblock Math. Scand. 96 (2005), no. 1, 63--95.

%\bibitem[Jun]{Jun}
%M. Junge.
%\newblock Fubini's theorem for ultraproducts of noncommmutative $L_p$-spaces.
%\newblock Canad. J. Math. 56 (2004), no. 5, 983--1021.
%
%\bibitem[Jun2]{Jun2}
%M. Junge.
%\newblock Fubini's theorem for ultraproducts of noncommmutative $L_p$-spaces II.
%\newblock Preprint.

\bibitem[JuX03]{JuX03}
M. Junge  and Q. Xu.
\newblock Noncommutative Burkholder/Rosenthal inequalities.
\newblock Ann. Probab. 31 (2003), no. 2, 948--995.

%\bibitem[KaR1]{KaR1}
%R. V. Kadison and J. R. Ringrose.
%\newblock Fundamentals of the theory of operator algebras. Vol. I. Elementary theory. Reprint of the 1983 original. 
%\newblock Graduate Studies in Mathematics, 15. American Mathematical Society, Providence, RI, 1997. 

\bibitem[KaR97]{KaR97}
R. V. Kadison and J. R. Ringrose.
\newblock Fundamentals of the theory of operator algebras. Vol. II. Advanced theory. Corrected reprint of the 1986 original.
\newblock Graduate Studies in Mathematics, 16. American Mathematical Society, Providence, RI, 1997. 

\bibitem[Kos14]{Kos14}
R. P. Kostecki.
\newblock $\mathrm{W}^*$-algebras and noncommutative integration.
\newblock Preprint, arXiv:1307.4818.


%\bibitem[KhS]{KhS}
%M. A. Khamsi and B. Sims.
%\newblock Ultra-methods in metric fixed point theory, Handbook of metric fixed point theory.
%\newblock Kluwer Acad. Publ., Dordrecht, 2001, pp.~177--199.

%\bibitem[Kir]{Kir}
%E. Kirchberg.
%\newblock On nonsemisplit extensions, tensor products and exactness of group C*-algebras.
%\newblock {\em Invent. Math.} 112, no. 3: 449--489, 1993.

%\bibitem[Kos2]{Kos2}
%H. Kosaki.
%\newblock Applications of the complex interpolation method to a von Neumann algebra: noncommutative $L^p$-spaces. 
%\newblock J. Funct. Anal. 56 (1984), no. 1, 29--78. 

\bibitem[Kos81]{Kos81}
H. Kosaki.
\newblock Positive cones and $L^p$-spaces associated with a von Neumann algebra. 
\newblock J. Operator Theory 6 (1981), no. 1, 13--23. 

%\bibitem[Knu]{Knu}
%S. Knudby.
%\newblock Semigroups of Herz-Schur multipliers.
%\newblock J. Funct. Anal. 266 (2014), no. 3, 1565--1610.

%\bibitem[KW]{KW}
%P. C. Kunstmann  and L. Weis.
%\newblock Maximal $L_p$-regularity for parabolic equations, Fourier multiplier theorems and $H^\infty$-functional calculus.
%\newblock pp. 65-311 in Functional analytic methods for evolution equations, Lect. Notes in Math. 1855, Springer (2004).



%\bibitem[LeR]{LeR}
%M. Levy and Y. Raynaud.
%\newblock Ultrapuissances des espaces $L^{p}(L^{q})$. (French) [Ultrapowers of $L^{p}(L^{q})$ spaces].
%\newblock C. R. Acad. Sci. Paris Sér. I Math. 299 (1984), no. 3, 81--84.

%\bibitem[Lab1]{Lab1}
%L. E. Labuschagne.
%\newblock Composition operators on non-commutative $L^p$-spaces.
%\newblock Exposition. Math. 17 (1999), no. 5, 429--467.

%\bibitem[Lab2]{Lab2}
%L. E. Labuschagne.
%\newblock A crossed product approach to Orlicz spaces.
%\newblock Proc. Lond. Math. Soc. (3) 107 (2013), no. 5, 965--1003.

%\bibitem[Lab3]{Lab3}
%L. E. Labuschagne.
%\newblock Multipliers on noncommutative Orlicz spaces.
%\newblock Quaest. Math. 37 (2014), no. 4, 531--546.

%\bibitem[Lac1]{Lac1}
%H. E. Lacey.
%\newblock The isometric theory of classical Banach spaces.
%\newblock Die Grundlehren der mathematischen Wissenschaften, Band 208. Springer-Verlag, New York-Heidelberg, 1974.

%\bibitem[LaM]{LaM}
%L. E. Labuschagne and W. A. Majewski.
%\newblock Maps on noncommutative Orlicz spaces.
%\newblock Illinois J. Math. 55 (2011), no. 3, 1053--1081 (2013).

\bibitem[LRR09]{LRR09}
C. Le Merdy, \'E. Ricard and J. Roydor.
\newblock Completely 1-complemented subspaces of Schatten spaces.
\newblock Trans. Amer. Math. Soc. 361 (2009), no. 2, 849--887.

\bibitem[Meg98]{Meg98}
R. E. Megginson.
\newblock An introduction to Banach space theory.
\newblock Graduate Texts in Mathematics, 183. Springer-Verlag, New York, 1998.

\bibitem[Mos06]{Mos06}
M.S. Moslehian.
\newblock A survey of the complemented subspaces problem.
\newblock Trends in Math. 9 (2006), no. 1, 91--98.

%\bibitem[Ter]{Ter}
%M. Terp.
%\newblock Interpolation spaces between a von Neumann algebra and its predual. .
%\newblock J. Operator Theory 8 (1982), no. 2, 327--360. .

\bibitem[NeR11]{NeR11}
M. Neal and B. Russo.
\newblock Existence of contractive projections on preduals of JBW*-triples.
\newblock Israel J. Math. 182 (2011), 293--331.

\bibitem[NO02]{NO02}
P. W. Ng, and N. Ozawa.
\newblock A characterization of completely 1-complemented subspaces of noncommutative $L^1$-spaces.
\newblock Pacific J. Math. 205 (2002), 171--195.

%\bibitem[Oza]{Oza}
%N. Ozawa.
%\newblock About the QWEP conjecture.
%\newblock Internat. J. Math.15, no. 5: 501--530, 2004.

\bibitem[Pal01]{Pal01}
T. W. Palmer.
\newblock Banach algebras and the general theory of $*$-algebras. Vol. 2. *-algebras.
\newblock Encyclopedia of Mathematics and its Applications, 79. Cambridge University Press, Cambridge, 2001.

%\bibitem[Pat18]{Pat18}
%H. K. Pathak.
%\newblock An Introduction to Nonlinear Analysis and Fixed Point Theory.
%\newblock Springer, Singapore, 2018.

\bibitem[Pau02]{Pau02}
V. Paulsen.
\newblock Completely bounded maps and operator algebras.
\newblock Cambridge University Press, Cambridge (2002).

\bibitem[PeT72]{PeT72}
G. K. Pedersen and M. Takesaki.
\newblock The operator equation $THT=K$.
\newblock Proc. Amer. Math. Soc. 36 (1972), 311--312.

\bibitem[Pis98]{Pis98}
G. Pisier.
\newblock Non-commutative vector valued $L_p$-spaces and completely $p$-summing maps.
\newblock Ast\'erisque, 247, 1998.

\bibitem[Pis03]{Pis03}
G. Pisier.
\newblock Introduction to operator space theory.
\newblock Cambridge University Press, Cambridge, 2003.

%\bibitem[Pis3]{Pis3}
%G. Pisier.
%\newblock Regular operators between non-commutative $L_p$-spaces.
%\newblock Bull. Sci. Math. 119 (1995), no. 2, 95--118.

%\bibitem[Pis7]{Pis7}
%G. Pisier.
%\newblock Complex interpolation between Hilbert, Banach and operator spaces.
%\newblock Mem. Amer. Math. Soc. 208 (2010), no. 978.
%
%\bibitem[Pis-IHP]{Pis-IHP}
%G. Pisier.
%\newblock Martingales in {B}anach spaces (in connection with {T}ype and
  %{C}otype).
%\newblock 2011, Course IHP.

\bibitem[PiX03]{PiX03}
G. Pisier and Q. Xu.
\newblock Non-commutative $L^p$-spaces.
\newblock 1459--1517 in Handbook of the Geometry of Banach Spaces, Vol. II, edited by W.B. Johnson and J. Lindenstrauss, Elsevier (2003).

%\bibitem[PoS1]{PoS1}
%D. Potapov and F. Sukochev.
%\newblock The Haar system in the preduals of hyperfinite factors.
%\newblock Canad. Math. Bull. 54 (2011), no. 2, 347--363.

%\bibitem[Pru]{Pru}
%Bebe Prunaru.
%\newblock On the Choi-Effros Multiplication.
%\newblock Arxiv.

%\bibitem[Rand1]{Rand1}
%B. Randrianantoanina.
%\newblock On isometric stability of complemented subspaces of $L_p$.
%\newblock Israel J. Math. 113 (1999), 45--60.

\bibitem[Ran01]{Rand01}
B. Randrianantoanina.
\newblock Norm-one projections in Banach spaces.
\newblock International Conference on Mathematical Analysis and its Applications (Kaohsiung, 2000). Taiwanese J. Math. 5 (2001), no. 1, 35--95.

%\bibitem[Rand3]{Rand3}
%B. Randrianantoanina.
%\newblock Contractive projections in Orlicz sequence spaces.
%\newblock Abstr. Appl. Anal. 2004, no. 2, 133--145.

%\bibitem[Ray1]{Ray1}
%Raynaud, Yves.
%\newblock On ultrapowers of non commutative $L_p$ spaces.
%\newblock J. Operator Theory 48 (2002), no. 1, 41--68.

\bibitem[Ray03]{Ray03}
Y. Raynaud.
\newblock $L_p$-spaces associated with a von Neumann algebra without trace: a gentle introduction via complex interpolation.
\newblock  Trends in Banach spaces and operator theory (Memphis, TN, 2001), 245--273, Contemp. Math., 321, Amer. Math. Soc., Providence, RI, 2003. 

\bibitem[Ray04]{Ray04}
Y. Raynaud.
\newblock The range of a contractive projection in $L_p(H)$.
\newblock Rev. Mat. Complut. 17 (2004), no. 2, 485--512. 

%\bibitem[Ray4]{Ray4}
%Y. Raynaud.
%\newblock Sur les sous-espaces de $L^p(L^q)$. (French) [On the subspaces of $L^p(L^q)$].
%\newblock Séminaire d'Analyse Fonctionelle 1984/1985, 49--71, Publ. Math. Univ. Paris VII, 26, Univ. Paris VII, Paris, 1986.

%\bibitem[Ray5]{Ray5}
%Y. Raynaud.
%\newblock Sous-espaces $l^r$ et géométrie des espaces $L^p(L^q)$ et $L^\phi$. (French) [$l^r$ subspaces and geometry of $L^p(L^q)$ and $L^\phi$ spaces].
%\newblock C. R. Acad. Sci. Paris Sér. I Math. 301 (1985), no. 6, 299--302.

%\bibitem[Ric]{Ric}
%\'E. Ricard.
%\newblock A Markov dilation for self-adjoint Schur multipliers.
%\newblock Proc. Amer. Math. Soc. 136 (2008), no. 12, 4365--4372.

%\bibitem[Sc]{Sc}
%E. Schechter.
%\newblock Handbook of analysis and its foundations.
%\newblock  Academic Press, Inc., San Diego, CA, 1999.



%\bibitem[RaX]{RaX}
%Y. Raynaud, Q. Xu.
%\newblock On subspaces of non-commutative $L_p$-spaces.
%\newblock J. Funct. Anal. 203 (2003), no. 1, 149-196.

%\bibitem[RoS1]{RoS1}
%D. W. Robinson and E. St\o rmer.
%\newblock Lie and Jordan structure in operator algebras.
%\newblock J. Austral. Math. Soc. Ser. A 29 (1980), no. 2, 129--142.

\bibitem[Sch86]{Sch86}
L. M. Schmitt.
\newblock The Radon-Nikodym theorem for $L^p$-spaces of W*-algebras.
\newblock Publ. Res. Inst. Math. Sci. 22 (1986), no. 6, 1025--1034.

\bibitem[See66]{See66}
G. L. Seever.
\newblock Nonnegative projections on $C_{0}(X)$.
\newblock Pacific J. Math. 17 (1966), 159--166.

%\bibitem[SkV1]{SkV1}
%A. Skalski and A. Viselter.
%\newblock Convolution semigroups on locally compact quantum groups and noncommutative Dirichlet forms.
%\newblock Preprint, arXiv:1709.04873.

%\bibitem[Sto1]{Sto1}
%E. St\o rmer.
%\newblock Conditional expectations and projection maps of von Neumann algebras.
%\newblock  Operator algebras and applications (Samos, 1996), 449--461,
%NATO Adv. Sci. Inst. Ser. C Math. Phys. Sci., 495, Kluwer Acad. Publ., Dordrecht, 1997.

\bibitem[Sto13]{Sto13}
E. St\o rmer.
\newblock Positive linear maps of operator algebras.
\newblock Springer Monographs in Mathematics. Springer, Heidelberg, 2013.

%\bibitem[Sto3]{Sto3}
%E. St\o rmer.
%\newblock Positive linear maps of operator algebras.
%\newblock Acta Math. 110 1963 233--278.

%\bibitem[Sto4]{Sto4}
%E. St\o rmer.
%\newblock Decomposable positive maps on C*-algebras.
%\newblock Proc. Amer. Math. Soc. 86 (1982), no. 3, 402--404.

%\bibitem[Sto5]{Sto5}
%E. St\o rmer.
%\newblock Jordan algebras of type I.
%\newblock Acta Math. 115 1966 165--184.

%\bibitem[Sto6]{Sto6}
%E. St\o rmer.
%\newblock A completely positive map associated with a positive map.
%\newblock Pacific J. Math. 252 (2011), no. 2, 487--492.

%\bibitem[Sto7]{Sto7}
%E. St\o rmer.
%\newblock A decomposition theorem for positive maps, and the projection onto a spin factor.
%\newblock Math. Scand. 118 (2016), no. 1, 106--118.

%\bibitem[Sto8]{Sto8}
%E. St\o rmer.
%\newblock Decomposition of positive projections on C*-algebras.
%\newblock Math. Ann. 247 (1980), no. 1, 21--41.

%\bibitem[Sto9]{Sto9}
%E. St\o rmer.
%\newblock Positive projections onto Jordan algebras and their enveloping von Neumann algebras.
%\newblock Ideas and methods in mathematical analysis, stochastics, and applications (Oslo, 1988), 389--393, Cambridge Univ. Press, Cambridge, 1992.

%\bibitem[Sto10]{Sto10}
%E. St\o rmer.
%\newblock Positive projections with contractive complements on C*-algebras.
%\newblock J. London Math. Soc. (2) 26 (1982), no. 1, 132--142.
%
%\bibitem[Sto11]{Sto11}
%E. St\o rmer.
%\newblock On the Jordan structure of C*-algebras.
%\newblock Trans. Amer. Math. Soc. 120 (1965), 438--447.

%\bibitem[Sto12]{Sto12}
%E. St\o rmer.
%\newblock On projection maps of von Neumann algebras.
%\newblock Math. Scand. 30 (1972), 46--50.
%
%\bibitem[Sto13]{Sto13}
%E. St\o rmer.
%\newblock Multiplicative properties of positive maps.
%\newblock Math. Scand. 100 (2007), no. 1, 184--192.

%\bibitem[Sto14]{Sto14}
%E. St\o rmer.
%\newblock Positive linear maps of C*-algebras. Foundations of quantum mechanics and ordered linear spaces.
%\newblock Advanced Study Inst., Marburg, 1973. Lecture Notes in Phys., Vol. 29, Springer, Berlin, 1974, 85--106

\bibitem[Str81]{Str81}
S. Stratila.
\newblock Modular theory in operator algebras.
\newblock Taylor and Francis, 1981.

%\bibitem[Str14]{Str14}
%M. Straatman.
%\newblock The duality map.
%\newblock Bachelorscriptie, Universiteit Leiden, 2014.

%\bibitem[StZ1]{StZ1}
%S. Stratila and L. Zsido.
%\newblock Lectures on von Neumann algebras. Revision of the 1975 original. Translated from the Romanian by Silviu Teleman.
%\newblock Editura Academiei, Bucharest; Abacus Press, Tunbridge Wells, 1979.


%\bibitem[Tak1]{Tak1}
%M. Takesaki.
%\newblock Theory of operator algebras. I. Reprint of the first (1979) edition.
%\newblock Encyclopaedia of Mathematical Sciences, 124. Operator Algebras and Non-commutative Geometry, 5. Springer-Verlag, Berlin, 2002.

%\bibitem[Tak1]{Tak1}
%M. Takesaki.
%\newblock Duality for crossed products and the structure of von Neumann algebras of type III.
%\newblock Acta Math. 131 (1973), 249--310.

\bibitem[Tak03]{Tak03}
M. Takesaki.
\newblock Theory of operator algebras. II.
\newblock Encyclopaedia of Mathematical Sciences, 125. Operator Algebras and Non-commutative Geometry, 6. Springer-Verlag, Berlin, 2003.

%\bibitem[Tak3]{Tak3}
%M. Takesaki.
%\newblock Conditional expectations in von Neumann algebras. 
%\newblock J. Functional Analysis {\bf 9} (1972), 306-321.

\bibitem[Ter81]{Terp81}
M. Terp.
\newblock $L^p$ spaces associated with von Neumann algebras.
\newblock Notes, Math. Institute, Copenhagen Univ., 1981.

%\bibitem[Tom59]{Tom1}
%J. Tomiyama.
%\newblock On the product projection of norm one in the direct product of operator algebras.
%\newblock T\^ohoku Math. J. (2) 11 (1959), 305--313.

\bibitem[Tza69]{Tza69}
L. Tzafriri.
\newblock Remarks on contractive projections in $L_{p}$-space.
\newblock Israel J. Math. 7 (1969), 9--15.

%\bibitem[Voi]{Voi}
%J. Voigt.
%\newblock Abstract Stein interpolation.
%\newblock Math. Nachr. 157 (1992), 197--199.

\bibitem[Wat88]{Wat88}
K. Watanabe.
\newblock Dual of noncommutative $L^p$-spaces with $0<p<1$.
\newblock Math. Proc. Cambridge Philos. Soc. 103 (1988), no. 3, 503--509.


%\bibitem[Wat2]{Wat2}
%K. Watanabe.
%\newblock On isometries between noncommutative $L^p$-spaces associated with arbitrary von Neumann algebras.
%\newblock J. Operator Theory 28 (1992), no. 2, 267--279.


%\bibitem[Wei]{Wei}
%M. Weigt.
%\newblock Jordan homomorphisms between algebras of measurable operators.
%\newblock Quaest. Math. 32 (2009), no. 2, 203--214.

%\bibitem[Wat]{Wat}
%S. Watanabe.
%\newblock $L_p$-continuity of positive semigroups on finite von Neumann algebras.
%\newblock Proc. Amer. Math. Soc. 102 (1988), no. 4, 840--842.

%\bibitem[Wul]{Wul}
%D. E. Wulbert.
%\newblock A note on the characterization of conditional expectation operators.
%\newblock Pacific J. Math. 34 (1970) 285--288.

\end{thebibliography}
\end{document}